\documentclass[11pt]{article}

\usepackage[a4paper,margin=1.05in]{geometry}
\usepackage[utf8]{inputenc}
\usepackage[T1]{fontenc}
\usepackage{lmodern}
\usepackage{amsmath,amssymb,amsthm,mathtools}
\usepackage{mathrsfs}
\usepackage{booktabs}
\usepackage{array}
\usepackage{enumitem}
\usepackage{microtype}
\usepackage[colorlinks=true,linkcolor=blue,citecolor=blue,urlcolor=blue]{hyperref}
\usepackage[nameinlink,capitalize,noabbrev]{cleveref}

\numberwithin{equation}{section}

\theoremstyle{plain}
\newtheorem{theorem}{Theorem}[section]
\newtheorem{proposition}[theorem]{Proposition}
\newtheorem{lemma}[theorem]{Lemma}
\newtheorem{corollary}[theorem]{Corollary}

\theoremstyle{definition}
\newtheorem{definition}[theorem]{Definition}
\newtheorem{assumption}[theorem]{Assumption}

\theoremstyle{remark}
\newtheorem{remark}[theorem]{Remark}

\newcommand{\NS}{\mathrm{NS}}
\newcommand{\loc}{\mathrm{loc}}
\newcommand{\cl}{\mathrm{cl}}
\newcommand{\intg}{\mathrm{int}}

\newcommand{\comp}{\mathrm{comp}}

\newcommand{\prs}{\mathrm{prs}}
\newcommand{\harm}{\mathrm{harm}}
\newcommand{\proj}{\mathrm{proj}}
\newcommand{\sync}{\mathrm{sync}}
\newcommand{\gate}{\mathrm{gate}}
\newcommand{\detc}{\mathrm{det}}
\newcommand{\chart}{\mathrm{chart}}
\newcommand{\locerr}{\mathrm{loc}}
\newcommand{\gap}{\mathrm{gap}}
\newcommand{\quot}{\mathrm{quot}}
\newcommand{\EF}{\mathrm{EF}}
\newcommand{\Err}{\mathsf{Err}}
\newcommand{\Tax}{\mathsf{Tax}}
\newcommand{\Leak}{\mathsf{Leak}}
\newcommand{\Dist}{\operatorname{dist}}
\newcommand{\dist}{\operatorname{dist}}

\newcommand{\supp}{\operatorname{supp}}

\newcommand{\A}{\mathcal A}

\newcommand{\F}{\mathcal F}
\newcommand{\K}{\mathcal K}
\newcommand{\Rprs}{\mathcal R_{\prs}}
\newcommand{\norm}[1]{\left\lVert #1\right\rVert}
\newcommand{\abs}[1]{\left\lvert #1\right\rvert}

\title{Finite-Window Recursive Audit Chains for Navier--Stokes Generated Packages}
\author{Runlong Yu\\
The University of Alabama, Tuscaloosa, AL, USA\\
\texttt{ryu5@ua.edu}}
\date{}

\begin{document}
\maketitle

\begin{abstract}
We develop a finite-window recursive audit framework for Navier--Stokes-generated packages. On a fixed window, the underlying anti-phantom certificate asserts that a baseline-visible package is either detected locally or charged to an explicit residual ledger. The first component of the framework is a broad one-step admissibility criterion for pressure-admissible finite-energy Navier--Stokes-generated audit packages under an explicit ledger of synchronization, localization, projection, harmonic-tail, chart, clean-gap, gate/slack, and detector mismatches. The second component is a finite-chain propagation theorem: once one-step admissibility supplies a finite renormalized chain of packages and a static finite-window audit certificate is available at each scale, variable-coefficient error recursion and weighted summation yield a recursive finite-window anti-phantom lower bound. We also give compact/effective pressure projection criteria, reduced quotient chart visibility, clean detector gaps, a pressure--flux--energy matrix kernel condition, a smooth reduced NS-generated verification class, and a conditional Caffarelli--Kohn--Nirenberg compatible defect-extraction criterion. 
\end{abstract}

\tableofcontents

\section{Introduction}

We work within the Leray--Hopf finite-energy framework for incompressible Navier--Stokes solutions \cite{Leray1934,Hopf1951} and the suitable-weak-solution partial regularity theory of Scheffer and Caffarelli--Kohn--Nirenberg, together with later refinements and expositions \cite{Scheffer1976,Scheffer1977,CKN1982,Lin1998,SereginLectureNotes}.  The pressure and localization modules use standard Calderon--Zygmund pressure decompositions, pressure regularity, and local energy methods \cite{SohrWahl1986,SereginSverak2002,ESS2003}.  The finite-window ledger, audit, and anti-phantom terminology follows the structural reduction framework developed in \cite{YuCriticalLedgers2026,YuSingularityAuditTransfer2026,YuComputationalAntiPhantom2026}.

The results are not a one-component regularity criterion.  Nevertheless, the CKN-compatible defect-extraction branch and the scale-critical defect-package language are adjacent to the one-component and degeneration literature \cite{KukavicaZiane2006,CaoTiti2011,CheminZhang2016,CheminZhangZhang2017,HanLeiLiZhao2019,KangNguyen2023,YuOneComponent2026,YuStrict2026,YuSchur2026,YuInvisible2026}.

The static finite-window audit theory \cite{YuSingularityAuditTransfer2026,YuComputationalAntiPhantom2026} proves a lower bound of the form
\begin{equation}\label{eq:static-audit-fw}
  M_{\Lambda}^{\loc}(D-\zeta_*)
  \ge
  c_{\Lambda,0}\,\dist_0(D,\Gamma)
  - E_{\Lambda,0}^{\quot}(D).
\end{equation}
The interpretation is that a defect visible in the baseline quotient cannot be simultaneously detector-silent and residual-cheap on the same finite window.  The natural next question is not yet an infinite-scale regularity question.  It is the finite recursive question: after one restriction-rescaling step, does an NS-generated package remain in the admissible class, and if so how do the audit coefficient and residual ledger change?

The analysis is organized around the interface
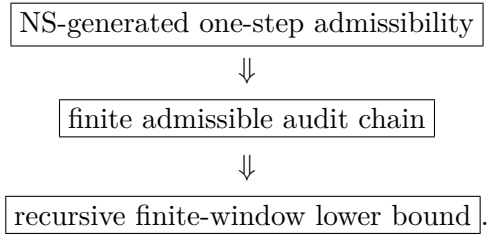
\[
\begin{array}{c}
  \boxed{\text{NS-generated one-step admissibility}}\\[0.35em]
  \Downarrow\\[0.35em]
  \boxed{\text{finite admissible audit chain}}\\[0.35em]
  \Downarrow\\[0.35em]
  \boxed{\text{recursive finite-window lower bound}}.
\end{array}
\]
The one-step half constructs and audits the map
\[
  D_k \longmapsto D_{k+1}=R_kD_k
\]
by restriction to the next geometric scale, Navier--Stokes rescaling to the unit cylinder, synchronization of the quotient representative, and recomputation of pressure, source, localization, gate/slack, detector, residual, and quotient coordinates.  The recursive half assumes a valid static certificate on each scale and sums those certificates with a variable-coefficient error recursion.

The main finite-chain output has the schematic form
\[
  \sum_{k=0}^K w_k M_k
  \ge
  c_K^{\min}\sum_{k=0}^K w_k\delta_k
  - E_K^{\mathrm{rec}},
\]
where the recursive error includes all explicitly charged one-step increments.  The theorem is useful precisely because the increment ledger is decomposed: PDE-generated channels such as pressure/source preservation, local energy/flux leakage, and localized Calderon--Zygmund pressure recomputation are separated from structural or reduced-model channels such as chart visibility, clean-gap kernel-freeness, and detector comparison.  The energy/flux bookkeeping is also close in spirit to local energy-transfer and anomalous-dissipation viewpoints \cite{ConstantinETiti1994,Eyink1994,DuchonRobert2000,BarkerPrange2021}.

\subsection{Main contributions}

The main contributions are as follows.
\begin{enumerate}[label=(\roman*)]
  \item It defines broad NS-generated finite-window audit packages and the renormalized one-step ledger.
  \item It proves fixed-window PDE modules: pressure/source preservation, energy/flux localization, and localized pressure recomputation.
  \item It records structural ledger insertions and reduced finite-dimensional detector, chart, clean-gap, and PFE matrix criteria.
  \item It proves finite-chain recursive audit propagation by variable-coefficient error recursion and weighted summation.
  \item It adds a conditional CKN-compatible defect-extraction branch showing how CKN-badness can force positive audit defect under an explicit audit-to-CKN metric comparison.
\end{enumerate}

\subsection{Scope of the results}

None of the theorems proves infinite-chain propagation, scale-uniform regularity, singularity exclusion, or the Navier--Stokes Clay problem.  The broad one-step theorem is conditional on every named mismatch being dominated by the corresponding ledger component.  The reduced chart and clean-gap mechanisms are finite-dimensional or compact-quotient inputs unless separately verified.  The CKN branch is a finite-window defect-extraction statement under an audit-to-CKN domination hypothesis; it is not a regularity proof.

\subsection{Logical status of the main assertions}

For clarity, we separate the assertions into three classes.  First, the fixed-window PDE estimates in \cref{sec:pde} are proved directly from pressure admissibility, the local energy inequality, scaling, Holder's inequality, and Calderon--Zygmund estimates.  Second, the synchronization, gate/slack, detector-comparison, coefficient-update, and recursive summation statements are finite-dimensional or algebraic consequences of explicitly stated hypotheses.  Third, compact chart visibility, clean detector gaps, PFE kernel-freeness, scale-uniform propagation, and the audit-to-CKN comparison are structural inputs unless a cited or separate theorem verifies them in a concrete class.  Thus each theorem should be read with its displayed hypotheses; no hidden claim is made that these structural hypotheses follow automatically for arbitrary suitable weak solutions.

\section{Preliminaries}
\label{sec:preliminaries}

\subsection{Normalized cylinders and scaling}

Fix \(0<\lambda<1\) and a scale sequence
\[
  r_k=\lambda^k r_0.
\]
The normalized cylinder is
\[
  Q_1=B_1\times(-1,0).
\]
For a local Navier--Stokes pair \((u,p)\) near \(z_k=(x_k,t_k)\), define
\[
  u_k(y,s)=r_k u(x_k+r_ky,t_k+r_k^2s),
  \qquad
  p_k(y,s)=r_k^2p(x_k+r_ky,t_k+r_k^2s).
\]
Unless explicitly stated, centers are fixed, \(z_k=z_0\).  Moving-center
variants require additional bookkeeping and are not used here.

\subsection{Pressure spaces}

The source and pressure observation spaces are
\[
  X_{\mathrm{src}}
  :=
  L^{3/2}((-1,0);L^{3/2}(B_1))^{3\times 3},
\]
\[
  Y_{\prs}
  :=
  L^{3/2}((-1,0);L^{3/2}(B_{1/2})),
  \qquad
  Y_{\harm}
  :=
  L^{3/2}((-1,0);L^{3/2}(B_{3/4})).
\]
The pressure map is
\[
  \Rprs F
  :=
  R_iR_j(F_{ij})\big|_{B_{1/2}},
\]
with zero extension outside \(B_1\).  The finite-window Calderon--Zygmund
bound is recorded as a standard singular-integral estimate and is used here in the local pressure-decomposition convention of Navier--Stokes theory \cite{SohrWahl1986,SereginLectureNotes}:
\begin{equation}\label{eq:cz-bound}
  \norm{\Rprs F}_{Y_{\prs}}
  \le C_{\mathrm{CZ}}\norm{F}_{X_{\mathrm{src}}}.
\end{equation}

\subsection{Cutoff convention}

Fix spatial cutoffs \(\eta,\chi\in C_c^\infty(B_1)\) with
\[
  \eta\equiv1\text{ on }B_{3/4},
  \qquad
  \chi\equiv1\text{ on }B_{3/4}.
\]
Let
\[
  A_\chi:=\operatorname{supp}\nabla\chi\cup\operatorname{supp}\Delta\chi.
\]
For the pressure cutoff, define the separated pressure annulus
\[
  A_\eta:=B_1\cap\operatorname{supp}(1-\eta).
\]
Since \(\eta\equiv1\) on \(B_{3/4}\), the sets \(B_{1/2}\) and \(A_\eta\)
have positive separation.
Fix a time cutoff \(\theta\in C_c^\infty((-1,0])\) and set
\[
  \phi(y,s)=\theta(s)\chi(y).
\]

\section{Broad NS-Generated Audit Packages}
\label{sec:packages}

\subsection{The broad package class}

\begin{definition}[Broad NS-generated scale package]\label{def:broad-package}
A scale-\(k\) package \(D_k(u,p)\) belongs to \(\A_k^{\NS}\) if it is generated
by normalized local data \((u_k,p_k)\) satisfying:
\begin{enumerate}[label=(\alph*)]
  \item pressure admissibility:
  \[
    u_k\in L^3(Q_1)^3,\qquad p_k\in L^{3/2}(Q_1),
  \]
  and
  \[
    -\Delta p_k=\partial_i\partial_j(u_{k,i}u_{k,j})
  \]
  in distributions on the spatial ball, modulo time-dependent constants;
  \item finite local energy:
  \[
    u_k\in L^\infty((-1,0);L^2(B_1))^3,\qquad
    \nabla u_k\in L^2(Q_1)^{3\times3};
  \]
  \item the local energy inequality on the normalized window, with admissible
  nonnegative test functions, as in the suitable weak solution framework \cite{CKN1982,SereginLectureNotes};
  \item finite audit coordinates
  \[
    D_k=(U_k,P_k^{\mathrm{act}},P_k^{\harm},S_k,L_k,G_k,
    \zeta_k,M_k,\delta_k,E_k);
  \]
  \item the baseline defect distance
  \[
    \delta_k:=\Dist_{0,k}(D_k,\Gamma_{\Lambda,k}^{\intg});
  \]
  \item fixed cutoff and pressure decomposition conventions as in
  \cref{sec:preliminaries};
  \item an admissible synchronized representative convention
  \(D_k-\zeta_k\);
  \item named channel budgets for localization, pressure projection, harmonic
  tail, chart, clean gap, gate/slack, detector, residual, and synchronization
  channels.
\end{enumerate}
\end{definition}

\begin{remark}[Status]
\Cref{def:broad-package} is extracted from the derivation plan.  It is a broad
finite-window class, not a smooth class.  It does not include scale-uniform
summability, detector kernel-freeness, chart visibility, or clean-gap
verification for arbitrary suitable weak solutions.
\end{remark}

\subsection{Pressure decomposition coordinates}

\begin{definition}[Active and harmonic pressure coordinates]
\label{def:pressure-coordinates}
For \(D_k\in\A_k^{\NS}\), define
\[
  F_{k,ij}^{\mathrm{act}}:=\eta\,u_{k,i}u_{k,j},
  \qquad
  \widetilde P_k^{\mathrm{act}}:=R_iR_j(F_{k,ij}^{\mathrm{act}}).
\]
The observed active pressure coordinate is
\[
  P_k^{\mathrm{act}}
  :=
  \widetilde P_k^{\mathrm{act}}\big|_{B_{1/2}}
  =
  \Rprs F_k^{\mathrm{act}},
\]
and
\[
  P_k^{\harm}
  :=
  \bigl(p_k-\widetilde P_k^{\mathrm{act}}\bigr)\big|_{B_{3/4}}
\]
is the harmonic pressure coordinate on the pressure-natural interior region.
\end{definition}

\begin{equation}\label{eq:delta-k}
  \delta_k:=\dist_{0,k}(D_k,\Gamma_{\Lambda,k}^{\intg}).
\end{equation}

\section{The One-Step Map and Ledger}
\label{sec:one-step}

\subsection{Renormalized one-step map}

\begin{definition}[One-step audit map]\label{def:one-step-map}
The one-step map
\[
  \mathcal R_k:D_k\mapsto D_{k+1}
\]
consists of:
\begin{enumerate}[label=(\roman*)]
  \item restriction from \(Q_k\) to \(Q_{k+1}\);
  \item Navier--Stokes rescaling to \(Q_1\);
  \item selection of a synchronized representative \(\zeta_{k+1}\);
  \item recomputation of active pressure, harmonic pressure, source,
  localization, gate/slack, detector, residual, and quotient-distance
  coordinates.
\end{enumerate}
\end{definition}

\subsection{One-step ledger}

\begin{definition}[One-step increment ledger]\label{def:one-step-ledger}
The one-step ledger is
\[
  \Delta_k
  =
  \Delta_k^{\sync}
  +\Delta_k^{\mathrm{loc},\EF}
  +\Delta_k^{\proj}
  +\Delta_k^{\harm}
  +\Delta_k^{\chart}
  +\Delta_k^{\gap}
  +\Delta_k^{\gate}
  +\Delta_k^{\detc}.
\]
Each component is required to dominate the mismatch of its corresponding
channel in the passage from \(D_k\) to \(D_{k+1}\).
\end{definition}

\begin{remark}[Energy/flux convention]
The localization convention used here is
\[
  \Delta_k^{\mathrm{loc}}
  :=
  \Delta_k^{\mathrm{loc},\EF}.
\]
Momentum-only and pressure-only leakage conventions are not used here.
\end{remark}

\subsection{Admissibility with error}

\begin{definition}[Next-scale admissibility with ledger]
\label{def:admissibility-with-ledger}
The notation
\[
  D_{k+1}\in\A_{k+1}^{\NS}(\Delta_k)
\]
means that the scale-\((k+1)\) package satisfies the defining entries of
\(\A_{k+1}^{\NS}\), with every one-step mismatch assigned to the corresponding
component of \(\Delta_k\).
\end{definition}

\begin{assumption}[One-step admissibility convention]\label{ass:one-step-adm}
The notation \(D_{k+1}\in\A_{k+1}^{\NS}(\Delta_k)\) means that the next-scale package satisfies the defining coordinate and structural entries of the broad NS-generated class with every named mismatch charged to the corresponding component of \(\Delta_k\).  The working one-step theorem, \cref{thm:main}, is the primary verification mechanism for this assumption in the NS-generated branch.
\end{assumption}

\section{PDE-Generated Modules}
\label{sec:pde}

\subsection{Pressure/source preservation}

\begin{proposition}[Pressure/source preservation]
\label{prop:pressure-source-preservation}
Let \(D_k\in\A_k^{\NS}\) be generated by pressure-admissible data.  Then
\[
  U_k=u_k\in L^3(Q_1)^3,
  \qquad
  F_k^{\mathrm{act}}\in X_{\mathrm{src}},
\]
\[
  P_k^{\mathrm{act}}\in Y_{\prs},
  \qquad
  P_k^{\harm}\in Y_{\harm}.
\]
Moreover \(P_k^{\harm}\) is harmonic in the interior region on which
\(\eta=1\), modulo the chosen pressure normalization.
\end{proposition}

\begin{proof}
Since \(u_k\in L^3(Q_1)^3\),
\[
  u_k\otimes u_k\in L^{3/2}(Q_1)^{3\times3},
  \qquad
  \norm{u_k\otimes u_k}_{L^{3/2}(Q_1)}
  \le
  \norm{u_k}_{L^3(Q_1)}^2.
\]
The cutoff \(\eta\) is smooth and bounded, hence
\[
  F_k^{\mathrm{act}}=\eta\,u_k\otimes u_k\in X_{\mathrm{src}}.
\]
By the fixed-window Calderon--Zygmund estimate \eqref{eq:cz-bound},
\[
  \norm{P_k^{\mathrm{act}}}_{Y_{\prs}}
  =
  \norm{\Rprs F_k^{\mathrm{act}}}_{Y_{\prs}}
  \le
  C_{\mathrm{CZ}}\norm{F_k^{\mathrm{act}}}_{X_{\mathrm{src}}},
\]
so \(P_k^{\mathrm{act}}\in Y_{\prs}\).

The global active pressure lift
\(\widetilde P_k^{\mathrm{act}}=R_iR_j(F_{k,ij}^{\mathrm{act}})\) is locally
\(L^{3/2}\) on \(B_{3/4}\) by the same Calderon--Zygmund bound applied before
restriction.  Since \(p_k\in L^{3/2}(Q_1)\), the difference
\[
  P_k^{\harm}
  =
  (p_k-\widetilde P_k^{\mathrm{act}})\big|_{B_{3/4}}
\]
belongs to \(Y_{\harm}\).

It remains only to identify the interior equation.  With the convention used
in \(\Rprs\),
\[
  -\Delta R_iR_jF_{ij}=\partial_i\partial_jF_{ij}
\]
in distributions.  Therefore
\[
  -\Delta \widetilde P_k^{\mathrm{act}}
  =
  \partial_i\partial_j(\eta u_{k,i}u_{k,j}).
\]
On every smaller interior ball \(B_\rho\Subset B_{3/4}\), the cutoff satisfies
\(\eta=1\), so
\[
  \partial_i\partial_j(\eta u_{k,i}u_{k,j})
  =
  \partial_i\partial_j(u_{k,i}u_{k,j})
\]
there.  Pressure admissibility gives
\[
  -\Delta p_k=\partial_i\partial_j(u_{k,i}u_{k,j}),
\]
and hence
\[
  -\Delta\bigl(p_k-\widetilde P_k^{\mathrm{act}}\bigr)=0
\]
in distributions on \(B_\rho\), for almost every time.  Since \(\rho<3/4\)
is arbitrary, \(P_k^{\harm}\) is spatially harmonic on the interior region
where \(\eta=1\), modulo the allowed time-dependent pressure normalization.
\end{proof}

\begin{remark}[Status]
\Cref{prop:pressure-source-preservation} is a proved finite-window module.
It does not assert compactness, scale-uniformity, or membership in any reduced
detector or chart class.
\end{remark}

\subsection{Energy/flux localization}

The localization estimates below are fixed-window consequences of the local energy inequality for suitable weak solutions \cite{Leray1934,Hopf1951,CKN1982}.

\begin{definition}[Energy/flux leakage]\label{def:ef-leakage}
Define
\[
\begin{aligned}
  \Leak_{\phi}^{\EF}(u,p)
  &:=
  \int_{Q_1}|u|^2\bigl(|\partial_t\phi|+|\Delta\phi|\bigr)\\
  &\quad+
  \int_{Q_1}\bigl(|u|^3+2|p||u|\bigr)|\nabla\phi|
  +
  2\int_{Q_1}|\nabla u|^2\phi .
\end{aligned}
\]
Let
\[
  \Omega_\chi:=\operatorname{supp}\chi.
\]
The pulled-back one-step energy/flux leakage in scale-\(k\) normalized
variables is
\[
\begin{aligned}
  \Leak_{k\to k+1}^{\EF}
  :=
  &\int_{-\lambda^2}^{0}\int_{\lambda\Omega_\chi}
  \bigl(|u_k|^2+|\nabla u_k|^2\bigr)\\
  &\quad+
  \int_{-\lambda^2}^{0}\int_{\lambda A_\chi}
  \bigl(|u_k|^2+|u_k|^3+|p_k||u_k|\bigr).
\end{aligned}
\]
\end{definition}

\begin{proposition}[Energy/flux localization]\label{prop:ef-localization}
Assume \(D_k\in\A_k^{\NS}\) is generated by finite-energy data satisfying the
local energy inequality on the normalized window.  Then the localization terms
created by restriction and rescaling to the next scale are charged by
\[
  \Leak_{\phi}^{\EF}(u_{k+1},p_{k+1})
  \le
  C_{\lambda,\chi,\theta}\Leak_{k\to k+1}^{\EF}.
\]
Consequently the localization component is admissible whenever
\[
  \Delta_k^{\mathrm{loc},\EF}
  \ge
  C_{\lambda,\chi,\theta}\Leak_{k\to k+1}^{\EF}.
\]
\end{proposition}

\begin{proof}
The local energy inequality is the admissibility mechanism that makes the
energy/flux channel a legitimate localization channel.  The estimate displayed
here is the fixed-window scaling bound for the terms in that channel.

Use the one-step rescaling
\[
  u_{k+1}(Y,S)=\lambda u_k(\lambda Y,\lambda^2S),
  \qquad
  p_{k+1}(Y,S)=\lambda^2p_k(\lambda Y,\lambda^2S).
\]
With \(y=\lambda Y\), \(s=\lambda^2S\), one has
\[
  dY\,dS=\lambda^{-5}\,dy\,ds.
\]
The energy term with the time derivative satisfies
\[
\begin{aligned}
  \int_{Q_1}|u_{k+1}|^2|\partial_S\phi|\,dY\,dS
  &\le
  \lambda^{-3}\norm{\partial_S\phi}_{L^\infty}
  \int_{-\lambda^2}^{0}\int_{\lambda\Omega_\chi}|u_k|^2\,dy\,ds .
\end{aligned}
\]
The energy term with the spatial Laplacian satisfies
\[
\begin{aligned}
  \int_{Q_1}|u_{k+1}|^2|\Delta_Y\phi|\,dY\,dS
  &\le
  \lambda^{-3}\norm{\Delta\chi}_{L^\infty}\norm{\theta}_{L^\infty}
  \int_{-\lambda^2}^{0}\int_{\lambda A_\chi}|u_k|^2\,dy\,ds .
\end{aligned}
\]
The cubic flux term obeys
\[
  \int_{Q_1}|u_{k+1}|^3|\nabla_Y\phi|\,dY\,dS
  \le
  \lambda^{-2}\norm{\nabla\chi}_{L^\infty}\norm{\theta}_{L^\infty}
  \int_{-\lambda^2}^{0}\int_{\lambda A_\chi}|u_k|^3\,dy\,ds ,
\]
and the pressure flux term obeys
\[
  \int_{Q_1}|p_{k+1}|\,|u_{k+1}|\,|\nabla_Y\phi|\,dY\,dS
  \le
  \lambda^{-2}\norm{\nabla\chi}_{L^\infty}\norm{\theta}_{L^\infty}
  \int_{-\lambda^2}^{0}\int_{\lambda A_\chi}|p_k|\,|u_k|\,dy\,ds .
\]
Finally,
\[
  \nabla_Yu_{k+1}(Y,S)
  =
  \lambda^2(\nabla_yu_k)(\lambda Y,\lambda^2S),
\]
so
\[
  \int_{Q_1}|\nabla_Yu_{k+1}|^2\phi\,dY\,dS
  \le
  \lambda^{-1}\norm{\phi}_{L^\infty}
  \int_{-\lambda^2}^{0}\int_{\lambda\Omega_\chi}|\nabla_yu_k|^2\,dy\,ds .
\]
Combining the five estimates and absorbing the fixed powers of \(\lambda\)
and cutoff norms into \(C_{\lambda,\chi,\theta}\) gives the claimed bound.
\end{proof}

\begin{remark}[Why the full cutoff support appears]
The terms involving \(|\partial_S\phi|\) and \(\phi|\nabla u|^2\) are supported
on the pulled-back support of \(\chi\), not only on the spatial transition
annulus \(A_\chi\).  A shell-only leakage estimate would require an additional
absorption, monotonicity, or shell-to-core argument.  No such smallness or
scale-uniform statement is claimed here.
\end{remark}

\subsection{Pressure recomputation and commutator leakage}

The commutator estimate uses only the separated-kernel part of the Calderon--Zygmund theory for the double Riesz transform.

\begin{proposition}[Localized pressure recomputation]
\label{prop:commutator}
For pressure-admissible data in \(\A_k^{\NS}\), the localized
Calderon--Zygmund commutator estimate is
\[
  \norm{[\eta,R_iR_j](u_{k,i}u_{k,j})}_{Y_{\prs}}
  \le
  C_\eta\norm{u_k\otimes u_k}_{L^{3/2}((-1,0);L^{3/2}(A_\eta))}.
\]
Under a finite amplitude bound \(\norm{u_k}_{L^3(Q_1)}\le M_U\),
\[
  \norm{[\eta,R_iR_j](u_{k,i}u_{k,j})}_{Y_{\prs}}
  \le
  C_\eta M_U
  \norm{u_k}_{L^3((-1,0)\times A_\eta)}.
\]
\end{proposition}

\begin{proof}
Let \(T_{ij}=R_iR_j\) and \(f_{ij}=u_{k,i}u_{k,j}\).  We use the convention
\[
  [\eta,T_{ij}]f_{ij}
  :=
  \eta\,T_{ij}f_{ij}-T_{ij}(\eta f_{ij}).
\]
On the pressure observation ball \(B_{1/2}\), \(\eta=1\).  Hence
\[
  [\eta,T_{ij}]f_{ij}
  =
  T_{ij}\bigl((1-\eta)f_{ij}\bigr)
  \qquad\text{on }B_{1/2}.
\]
The function \((1-\eta)f_{ij}\) is supported in \(A_\eta\).  Since
\(\operatorname{dist}(B_{1/2},A_\eta)>0\), the kernel \(K_{ij}\) of \(T_{ij}\)
is smooth and bounded on \(B_{1/2}\times A_\eta\).  Therefore, for almost
every time \(s\),
\[
  |[\eta,T_{ij}]f_{ij}(x,s)|
  \le
  C_\eta\int_{A_\eta}|f(y,s)|\,dy,
  \qquad x\in B_{1/2}.
\]
Taking the \(L^{3/2}(B_{1/2})\)-norm and using the finite measure of
\(A_\eta\) gives
\[
  \norm{[\eta,T_{ij}]f_{ij}(\cdot,s)}_{L^{3/2}(B_{1/2})}
  \le
  C_\eta\norm{f(\cdot,s)}_{L^{3/2}(A_\eta)}.
\]
Taking the \(L^{3/2}\)-norm in time yields the claimed \(Y_{\prs}\) estimate.

For \(f=u_k\otimes u_k\),
\[
  \norm{u_k\otimes u_k}_{L^{3/2}((-1,0)\times A_\eta)}
  =
  \norm{u_k}_{L^3((-1,0)\times A_\eta)}^2.
\]
If \(\norm{u_k}_{L^3(Q_1)}\le M_U\), then
\[
  \norm{u_k}_{L^3((-1,0)\times A_\eta)}^2
  \le
  M_U\norm{u_k}_{L^3((-1,0)\times A_\eta)},
\]
which gives the finite-amplitude form.
\end{proof}

\section{Structural Ledger Assignments}
\label{sec:structural}

\subsection{Synchronization}

\begin{definition}[Broad near-minimizer synchronization loss]\label{def:broad-sync}
For each scale \(k\), assume the gauge class
\(\Gamma_{\Lambda,k}^{\intg}\) is nonempty and that
\[
  \delta_k
  :=
  \Dist_{0,k}(D_k,\Gamma_{\Lambda,k}^{\intg})<\infty .
\]
Fix \(\varepsilon_k>0\) and choose an \(\varepsilon_k\)-near minimizing
representative
\[
  \zeta_k\in\Gamma_{\Lambda,k}^{\intg},
  \qquad
  \lvert D_k-\zeta_k\rvert_{0,k}
  \le
  \delta_k+\varepsilon_k.
\]
Let
\[
  T_k^\Gamma:
  \Gamma_{\Lambda,k}^{\intg}\to\Gamma_{\Lambda,k+1}^{\intg}
\]
be the chosen finite-window gauge-transport map, and let
\(d_{\Gamma,k+1}\) be the gauge discrepancy functional at scale \(k+1\).
The broad synchronization error is
\[
  \mathrm{SyncErr}_k
  :=
  d_{\Gamma,k+1}\bigl(\zeta_{k+1},
  T_k^\Gamma\zeta_k\bigr)
  +
  \varepsilon_k+\varepsilon_{k+1}.
\]
\end{definition}

\begin{proposition}[Broad synchronization ledger insertion]
\label{prop:broad-sync-insertion}
Under the convention in \cref{def:broad-sync}, if
\[
  \Delta_k^{\sync}\ge \mathrm{SyncErr}_k,
\]
then the synchronization entry of the one-step ledger is verified.
\end{proposition}

\begin{proof}
The synchronization entry asks that the mismatch between the next selected
representative and the transported previous representative be charged to the
synchronization component of the ledger.  By definition, this mismatch and the
two near-minimizer tolerances are exactly \(\mathrm{SyncErr}_k\).  The
assumption \(\Delta_k^{\sync}\ge\mathrm{SyncErr}_k\) therefore assigns the
entire synchronization loss to the synchronization ledger component.
\end{proof}

\begin{remark}[Status of the synchronization convention]
\Cref{def:broad-sync,prop:broad-sync-insertion} give a broad quotient
ledger convention.  They do not prove uniform synchronization stability and
do not prove an estimate of the form
\[
  \lvert \zeta_{k+1}-R_k\zeta_k\rvert
  \le
  C E_k+C\Delta_k
\]
without additional assumptions.  That stronger estimate belongs to a later
scale-uniform branch.  An exact reduced finite-dimensional splitting model is
a separate reduced-model variant and is not used in the general finite-window theorem.
\end{remark}

\begin{assumption}[Synchronization ledger assignment under the broad convention]\label{ass:sync}
The representative selected at the next scale satisfies
\[
  \mathrm{SyncErr}_k
  \le
  \Delta_k^{\sync}.
\]
Here this is interpreted through the broad synchronization loss convention
in \cref{def:broad-sync}.  A reduced splitting estimate belongs to a
separate reduced-model variant.
\end{assumption}

\subsection{Projection and harmonic tails}

\begin{definition}[Projection and harmonic tail ledger datum]
\label{def:proj-harm-datum}
At the next scale let \(\mathsf P_{k+1}^{\cl}\in Y_{\prs}\) denote the
selected clean pressure observation whose finite-rank pressure projection is
recorded in the package, and let
\(\mathsf H_{k+1}^{\harm}\in Y_{\harm}\) denote the selected harmonic
pressure-tail observation.  Choose finite-rank maps
\[
  P_{N,k+1}^{\cl}:Y_{\prs}\to Y_{\prs},
  \qquad
  H_{M,k+1}^{\harm}:Y_{\harm}\to Y_{\harm}.
\]
No orthogonality is assumed unless it is specified as part of a later reduced
model.  Define the one-step projection and harmonic tail errors by
\[
  \mathrm{ProjErr}_k
  :=
  \norm{(I-P_{N,k+1}^{\cl})\mathsf P_{k+1}^{\cl}}_{Y_{\prs}},
\]
and
\[
  \mathrm{HarmErr}_k
  :=
  \norm{(I-H_{M,k+1}^{\harm})\mathsf H_{k+1}^{\harm}}_{Y_{\harm}}.
\]
\end{definition}

\begin{proposition}[Projection and harmonic ledger insertion]
\label{prop:proj-harm-insertion}
Under the finite-window datum in \cref{def:proj-harm-datum}, if
\[
  \Delta_k^{\proj}\ge \mathrm{ProjErr}_k,
  \qquad
  \Delta_k^{\harm}\ge \mathrm{HarmErr}_k,
\]
then the pressure-projection and harmonic-tail entries of the one-step
admissibility checklist are verified.
\end{proposition}

\begin{proof}
The pressure-projection entry asks that the portion of the selected clean
pressure observation outside the chosen finite-rank pressure coordinate be
charged to the projection ledger.  By \cref{def:proj-harm-datum}, that portion
has size \(\mathrm{ProjErr}_k\).  The harmonic-tail entry is identical with
\(\mathrm{HarmErr}_k\) for the chosen harmonic finite-rank map.  The two
displayed inequalities therefore assign both tails to their respective
ledger components.
\end{proof}

\begin{assumption}[Projection and harmonic ledger assignment]\label{ass:proj-harm}
The clean pressure projection and harmonic-tail mismatches are interpreted
through \cref{def:proj-harm-datum} and satisfy
\[
  \mathrm{ProjErr}_k\le\Delta_k^{\proj},
  \qquad
  \mathrm{HarmErr}_k\le\Delta_k^{\harm}.
\]
Compactness of a selected pressure image, effective finite-rank projection, or
explicit harmonic truncation estimates may verify these inequalities in
particular settings.  The statement records only the finite-window ledger
convention.
\end{assumption}

\begin{remark}[Status of projection and harmonic tails]
\Cref{def:proj-harm-datum,prop:proj-harm-insertion} do not prove convergence
as \(N,M\to\infty\), compactness of a pressure image, summability of tail
increments, or scale-uniform approximation.  They provide the finite-window
bookkeeping convention used by the one-step theorem.
\end{remark}

\subsection{Gate/slack channel}

\begin{definition}[Gate/slack ledger datum]\label{def:gate-slack-datum}
At the next scale let \(\mathsf G_{k+1}\in\mathcal Y_{\gate,k+1}\) be the
selected gate/slack coordinate in a finite-dimensional normed space, and let
\[
  \mathcal K_{\gate,k+1}\subset\mathcal Y_{\gate,k+1}
\]
be the admissible closed gate set.  Define the gate/slack violation by
\[
  \mathrm{GateErr}_k
  :=
  \Dist_{\mathcal Y_{\gate,k+1}}
  \bigl(\mathsf G_{k+1},\mathcal K_{\gate,k+1}\bigr).
\]
For scalar inequality gates this is the same convention as measuring the
positive part of the violated slack vector, after choosing the corresponding
finite-dimensional norm.
\end{definition}

\begin{proposition}[Gate/slack ledger insertion]\label{prop:gate-slack-insertion}
Under the finite-window datum in \cref{def:gate-slack-datum}, if
\[
  \Delta_k^{\gate}\ge\mathrm{GateErr}_k,
\]
then the gate/slack entry of the one-step admissibility checklist is verified.
\end{proposition}

\begin{proof}
The gate/slack entry asks that the next-scale gate coordinate lie in the
admissible gate set up to a charged violation.  By
\cref{def:gate-slack-datum}, this violation is exactly
\(\mathrm{GateErr}_k\).  If \(\Delta_k^{\gate}\) dominates it, the full
gate/slack mismatch is assigned to the gate component of the ledger.
\end{proof}

\begin{remark}[Status of gate/slack bookkeeping]
\Cref{def:gate-slack-datum,prop:gate-slack-insertion} do not prove that the
gate violation is small, summable, or generated by the Navier--Stokes
equations.  They only fix the finite-window convention used by the one-step
ledger.
\end{remark}

\subsection{Reduced detector comparison for PFE clean audit channels}

This subsection realizes the detector channel in a concrete reduced
finite-dimensional model.  The comparison is between a local reduced detector
and the pressure--flux--energy clean detector after applying the reduced clean
chart.

\begin{definition}[Reduced PFE detector-comparison datum]
\label{def:reduced-pfe-detector-datum}
Fix finite-dimensional normed spaces
\[
  X_{N,k}^{\loc},
  \qquad
  Y_{N,k}^{\cl},
  \qquad
  Z_{\Lambda,k}^{\loc},
  \qquad
  Z_{\Lambda,k}^{\mathrm{PFE}}.
\]
Let
\[
  \Theta_{\Lambda,k}^{N}:X_{N,k}^{\loc}\to Y_{N,k}^{\cl}
\]
be the reduced clean chart.  Let
\[
  T_{\Lambda,k}^{\loc}:X_{N,k}^{\loc}\to Z_{\Lambda,k}^{\loc},
  \qquad
  T_{\Lambda,k}^{\mathrm{PFE}}:
  Y_{N,k}^{\cl}\to Z_{\Lambda,k}^{\mathrm{PFE}}
\]
be the reduced local detector/tax map and the reduced clean PFE detector/tax
map.  Finally, let
\[
  J_{\Lambda,k}:Z_{\Lambda,k}^{\loc}\to Z_{\Lambda,k}^{\mathrm{PFE}}
\]
be the observation transport map.  For \(D\in X_{N,k}^{\loc}\) and
\(d\in Y_{N,k}^{\cl}\), define
\[
  M_{\Lambda,k}^{\loc}(D)
  :=
  \norm{T_{\Lambda,k}^{\loc}D}_{Z_{\Lambda,k}^{\loc}},
  \qquad
  M_{\Lambda,k}^{\mathrm{PFE}}(d)
  :=
  \norm{T_{\Lambda,k}^{\mathrm{PFE}}d}_{Z_{\Lambda,k}^{\mathrm{PFE}}}.
\]
The reduced detector commutator is
\[
  \mathcal C_{\mathrm{det},k}(D)
  :=
  T_{\Lambda,k}^{\mathrm{PFE}}\Theta_{\Lambda,k}^{N}(D)
  -
  J_{\Lambda,k}T_{\Lambda,k}^{\loc}(D).
\]
\end{definition}

\begin{definition}[Detector residual domination]
\label{def:reduced-detector-residual}
The concrete detector-commutator residual is
\[
  \operatorname{Err}_{\mathrm{det.red},k}(D)
  :=
  \norm{\mathcal C_{\mathrm{det},k}(D)}_{Z_{\Lambda,k}^{\mathrm{PFE}}}.
\]
If a broader reduced residual functional \(\operatorname{Err}_{\mathrm{red},k}\)
is used, the detector-comparison convention requires constants
\(C_{\mathrm{res},k}\ge0\) and \(\Delta_{\mathrm{det},k}\ge0\) such that
\[
  \operatorname{Err}_{\mathrm{det.red},k}(D)
  \le
  C_{\mathrm{res},k}\operatorname{Err}_{\mathrm{red},k}(D)
  +
  \Delta_{\mathrm{det},k}.
\]
With the concrete choice
\(\operatorname{Err}_{\mathrm{red},k}=\operatorname{Err}_{\mathrm{det.red},k}\),
one may take \(C_{\mathrm{res},k}=1\) and
\(\Delta_{\mathrm{det},k}=0\).
\end{definition}

\begin{assumption}[Observation transport lower bound]
\label{ass:detector-observation-transport}
There are constants \(c_{J,k}>0\) and \(\Delta_{J,k}\ge0\) such that, for all
\(D\in X_{N,k}^{\loc}\),
\[
  \norm{T_{\Lambda,k}^{\loc}D}_{Z_{\Lambda,k}^{\loc}}
  \ge
  c_{J,k}
  \norm{J_{\Lambda,k}T_{\Lambda,k}^{\loc}D}_{Z_{\Lambda,k}^{\mathrm{PFE}}}
  -
  \Delta_{J,k}.
\]
This is a reduced finite-dimensional observability or transport convention.
\end{assumption}

\begin{theorem}[Reduced detector comparison for PFE clean audit channels]
\label{thm:reduced-pfe-detector-comparison}
Assume the reduced detector-comparison datum in
\cref{def:reduced-pfe-detector-datum}, the residual domination in
\cref{def:reduced-detector-residual}, and the observation transport lower
bound in \cref{ass:detector-observation-transport}.  Then for every
\(D\in X_{N,k}^{\loc}\),
\[
  M_{\Lambda,k}^{\loc}(D)
  \ge
  c_{J,k}
  M_{\Lambda,k}^{\mathrm{PFE}}
  \bigl(\Theta_{\Lambda,k}^{N}D\bigr)
  -
  c_{J,k}C_{\mathrm{res},k}
  \operatorname{Err}_{\mathrm{red},k}(D)
  -
  c_{J,k}\Delta_{\mathrm{det},k}
  -
  \Delta_{J,k}.
\]
\end{theorem}

\begin{proof}
By definition of the detector commutator,
\[
  T_{\Lambda,k}^{\mathrm{PFE}}\Theta_{\Lambda,k}^{N}(D)
  =
  J_{\Lambda,k}T_{\Lambda,k}^{\loc}(D)
  +
  \mathcal C_{\mathrm{det},k}(D).
\]
Hence
\[
  \norm{J_{\Lambda,k}T_{\Lambda,k}^{\loc}(D)}_{Z_{\Lambda,k}^{\mathrm{PFE}}}
  \ge
  M_{\Lambda,k}^{\mathrm{PFE}}(\Theta_{\Lambda,k}^{N}D)
  -
  \norm{\mathcal C_{\mathrm{det},k}(D)}_{Z_{\Lambda,k}^{\mathrm{PFE}}}.
\]
Applying the observation transport lower bound and then the reduced residual
domination gives
\[
  M_{\Lambda,k}^{\loc}(D)
  \ge
  c_{J,k}
  M_{\Lambda,k}^{\mathrm{PFE}}(\Theta_{\Lambda,k}^{N}D)
  -
  c_{J,k}
  \operatorname{Err}_{\mathrm{det.red},k}(D)
  -
  \Delta_{J,k}
\]
and therefore the displayed estimate.
\end{proof}

\begin{corollary}[Normalized reduced detector comparison]
\label{cor:normalized-reduced-detector-comparison}
In the normalized observation case
\[
  c_{J,k}=1,
  \qquad
  \Delta_{J,k}=0,
\]
the comparison becomes
\[
  M_{\Lambda,k}^{\loc}(D)
  \ge
  M_{\Lambda,k}^{\mathrm{PFE}}
  \bigl(\Theta_{\Lambda,k}^{N}D\bigr)
  -
  C_{\mathrm{dc},k}
  \operatorname{Err}_{\mathrm{red},k}(D)
  -
  \Delta_{\mathrm{det},k},
\]
with \(C_{\mathrm{dc},k}=C_{\mathrm{res},k}\).
\end{corollary}

\begin{proof}
This is \cref{thm:reduced-pfe-detector-comparison} with
\(c_{J,k}=1\), \(\Delta_{J,k}=0\), and
\(C_{\mathrm{dc},k}=C_{\mathrm{res},k}\).
\end{proof}

\begin{corollary}[Detector ledger insertion for reduced PFE channels]
\label{cor:reduced-pfe-detector-ledger}
Let \(D_k^{N,\loc}\in X_{N,k}^{\loc}\) be the reduced local coordinate of the
synchronized package \(D_k-\zeta_k\).  Define
\[
  \mathrm{DetErr}_k^{\mathrm{PFE}}
  :=
  c_{J,k}C_{\mathrm{res},k}
  \operatorname{Err}_{\mathrm{red},k}(D_k^{N,\loc})
  +
  c_{J,k}\Delta_{\mathrm{det},k}
  +
  \Delta_{J,k}.
\]
If
\[
  \Delta_k^{\detc}
  \ge
  \mathrm{DetErr}_k^{\mathrm{PFE}},
\]
then the reduced PFE detector-comparison entry of the one-step admissibility
checklist is verified:
\[
  M_{\Lambda,k}^{\loc}(D_k^{N,\loc})
  \ge
  c_{J,k}
  M_{\Lambda,k}^{\mathrm{PFE}}
  \bigl(\Theta_{\Lambda,k}^{N}D_k^{N,\loc}\bigr)
  -
  \Delta_k^{\detc}.
\]
\end{corollary}

\begin{proof}
Substitute \(D=D_k^{N,\loc}\) in
\cref{thm:reduced-pfe-detector-comparison}.  The definition of
\(\mathrm{DetErr}_k^{\mathrm{PFE}}\) collects exactly the residual,
commutator, and observation-transport losses.  If
\(\Delta_k^{\detc}\) dominates this quantity, the displayed detector
comparison follows.
\end{proof}

\begin{remark}[Scope of the reduced detector comparison]
\label{rem:reduced-detector-scope}
\Cref{thm:reduced-pfe-detector-comparison,cor:reduced-pfe-detector-ledger}
are reduced finite-dimensional detector-comparison statements.  They do not
prove detector comparison for the full infinite-dimensional Navier--Stokes
package geometry.  They also do not prove scale-uniformity, singularity
extraction, Navier--Stokes regularity, or any Clay-level conclusion.
\end{remark}

\subsection{Chart and clean-gap channels}

\begin{assumption}[Reduced ledger assignment]\label{ass:reduced-detector}
The chart and clean-gap mismatches satisfy
\[
  \mathrm{ChartErr}_k\le\Delta_k^{\chart},\qquad
  \mathrm{GapErr}_k\le\Delta_k^{\gap}.
\]
These are structural inputs unless a reduced chart or clean detector model is
supplied.
\end{assumption}

\begin{remark}[Status of structural inputs]
\Cref{ass:sync,ass:proj-harm,ass:reduced-detector} and the gate/slack datum
in \cref{def:gate-slack-datum} are structural inputs extracted from the
derivation plan.  They are not asserted to follow from the Navier--Stokes
equations alone.
\end{remark}

\section{Finite-Window One-Step Admissibility Criteria}

The one-step admissibility assumption in \cref{ass:one-step-adm} is the first
recursive structural input.  This section records a checkable finite-window
criterion for it and verifies the pressure/source part for
Navier--Stokes-generated packages.  The result is intentionally local: it
does not prove detector stability, chart kernel-freeness, gate/slack closure,
or scale-uniformity.

\subsection{Coordinate admissibility budgets}

Use the normalized spaces
\[
  X_u:=L^3(Q_1)^3,\qquad
  X_{\mathrm{src}}
  :=
  L^{3/2}((-1,0);L^{3/2}(B_1))^{3\times3},
\]
\[
  Y_{\prs}:=L^{3/2}((-1,0);L^{3/2}(B_{1/2})),
  \qquad
  Y_{\harm}:=L^{3/2}((-1,0);L^{3/2}(B_{3/4})).
\]

\begin{definition}[Coordinate admissibility checklist]\label{def:coord-check}
At scale \(k+1\), a recomputed package \(D_{k+1}\) satisfies the coordinate
admissibility checklist with excess vector
\[
  \mathbf e_{k+1}
  =
  (e_u,e_{\mathrm{src}},e_{\prs},e_{\harm},
  e_{\sync},e_{\locerr},e_{\gate},e_{\detc},e_{\mathrm{chart}},e_{\gap})
\]
if the following hold:
\begin{enumerate}[label=(\alph*)]
  \item \(U_{k+1}\in X_u\), the selected source coordinate
  \(S_{k+1}\in X_{\mathrm{src}}\),
  \(P_{k+1}^{\mathrm{act}}\in Y_{\prs}\), and
  \(P_{k+1}^{\harm}\in Y_{\harm}\);
  \item the representative mismatch, localization leakage, gate/slack
  mismatch, detector mismatch, chart error, and clean-gap error are bounded
  by the corresponding entries of \(\mathbf e_{k+1}\);
  \item under the energy/flux convention used here, the
  localization entry is \(e_{\locerr}=e_{\locerr}^{\mathrm{EF}}\) and is
  verified by the update mechanism in \cref{thm:ef-loc-update};
  \item the pressure projection and harmonic truncation errors obey
  \[
    \mathrm{ProjErr}_{k+1}\le e_{\prs}+e_{\mathrm{src}},
    \qquad
    \mathrm{HarmErr}_{k+1}\le e_{\harm}.
  \]
  In a componentwise ledger these errors are later charged by requiring
  \(\Delta_k^{\proj}\ge\mathrm{ProjErr}_{k+1}\) and
  \(\Delta_k^{\harm}\ge\mathrm{HarmErr}_{k+1}\).
\end{enumerate}
The total checklist excess is
\[
  |\mathbf e_{k+1}|_1
  :=
  e_u+e_{\mathrm{src}}+e_{\prs}+e_{\harm}
  +e_{\sync}+e_{\locerr}+e_{\gate}+e_{\detc}
  +e_{\mathrm{chart}}+e_{\gap}.
\]
\end{definition}

\begin{theorem}[Coordinate-budget one-step admissibility]\label{thm:coord-adm}
Let \(D_k\in\mathcal A_k\) and let \(D_{k+1}=\mathcal R_k(D_k)\) be
obtained by restriction, rescaling, synchronization, and recomputation.  If
\(D_{k+1}\) satisfies the coordinate admissibility checklist with excess
\(\mathbf e_{k+1}\), with the localization entry supplied by the
energy/flux update of \cref{thm:ef-loc-update}.  Assume that the componentwise
ledger dominates the corresponding checklist errors, namely
\[
\begin{gathered}
  \Delta_k^{\sync}\ge e_{\sync},\qquad
  \Delta_k^{\mathrm{loc},\EF}\ge e_{\locerr},\qquad
  \Delta_k^{\gate}\ge e_{\gate},\qquad
  \Delta_k^{\detc}\ge e_{\detc},\\
  \Delta_k^{\chart}\ge e_{\mathrm{chart}},
  \qquad
  \Delta_k^{\gap}\ge e_{\gap},
  \qquad
  \Delta_k^{\proj}\ge e_{\prs}+e_{\mathrm{src}},
  \qquad
  \Delta_k^{\harm}\ge e_{\harm}.
\end{gathered}
\]
Equivalently, under a total-budget convention, it is enough to choose
\[
  \Delta_k\ge C_{\mathrm{adm}}|\mathbf e_{k+1}|_1
\]
for a fixed finite-window bookkeeping constant \(C_{\mathrm{adm}}\ge1\) and a
fixed allocation of the total budget to the named components.  Then
\(D_{k+1}\) belongs to the increment-enlarged admissible class
\(\mathcal A_{k+1}(\Delta_k)\).  In particular, if all checklist excesses
vanish, then \(D_{k+1}\in\mathcal A_{k+1}\) in the exact finite-window
sense.
\end{theorem}

\begin{proof}
The admissible class \(\mathcal A_{k+1}(\Delta_k)\) is defined by the same
coordinate membership requirements as \(\mathcal A_{k+1}\), with each
non-exact channel allowed an error bounded by its assigned component of the
ledger in \cref{def:one-step-ledger}.  The checklist gives the required
coordinate membership in \(X_u\), \(X_{\mathrm{src}}\), \(Y_{\prs}\), and
\(Y_{\harm}\), and assigns every synchronization, localization, gate/slack,
detector, chart, projection, harmonic, and clean-gap defect to one of the
named components.  The displayed componentwise inequalities dominate those
errors by the corresponding ledger entries.  Under the coarser total-budget
convention, the fixed allocation and the bound
\(\Delta_k\ge C_{\mathrm{adm}}|\mathbf e_{k+1}|_1\) provide the same component
bounds after increasing \(C_{\mathrm{adm}}\) if necessary.  This is precisely
membership in \(\mathcal A_{k+1}(\Delta_k)\).  If
\(\mathbf e_{k+1}=0\), no enlargement is used.
\end{proof}

\subsection{A reduced synchronization model}

The synchronization increment \(\Delta_k^{\sync}\) measures the failure of
the representative chosen after renormalization to equal the renormalized
representative chosen before renormalization.  In full package geometry this
may be a genuine structural input.  In a reduced finite-dimensional quotient
model, however, the drift has a simple algebraic form.

\begin{definition}[Reduced synchronization datum]\label{def:sync-datum}
A reduced synchronization datum consists of finite-dimensional normed spaces
\(\mathfrak X_k\), gauge subspaces \(\mathfrak Z_k\subset\mathfrak X_k\),
closed complements \(\mathfrak H_k\), and projections
\[
  \Pi_k^{\mathfrak Z}:\mathfrak X_k\to\mathfrak Z_k,
  \qquad
  \Pi_k^{\mathfrak H}:\mathfrak X_k\to\mathfrak H_k,
  \qquad
  \mathfrak X_k=\mathfrak Z_k\oplus\mathfrak H_k.
\]
The synchronized representative is
\[
  \zeta_k(D):=\Pi_k^{\mathfrak Z}D,
  \qquad
  D-\zeta_k(D)=\Pi_k^{\mathfrak H}D.
\]
For a reduced one-step map \(R_k^N:\mathfrak X_k\to\mathfrak X_{k+1}\), the
synchronization drift is
\[
  \Delta_{k}^{\sync,N}(D)
  :=
  \norm{\zeta_{k+1}(R_k^ND)-R_k^N\zeta_k(D)}_{\mathfrak X_{k+1}}.
\]
\end{definition}

\begin{theorem}[Finite-dimensional synchronization drift]\label{thm:sync-drift}
Let \(R_k^N:\mathfrak X_k\to\mathfrak X_{k+1}\) be linear and gauge
compatible:
\[
  R_k^N\mathfrak Z_k\subset \mathfrak Z_{k+1}.
\]
Then
\[
  \Delta_{k}^{\sync,N}(D)
  \le
  \norm{\Pi_{k+1}^{\mathfrak Z}R_k^N\Pi_k^{\mathfrak H}}_{\mathfrak X_k\to
  \mathfrak X_{k+1}}
  \,
  \norm{D-\zeta_k(D)}_{\mathfrak X_k}.
\]
If, in addition,
\[
  R_k^N\mathfrak H_k\subset\mathfrak H_{k+1},
\]
then \(\Delta_{k}^{\sync,N}(D)=0\) for every \(D\).
\end{theorem}

\begin{proof}
Decompose \(D=\Pi_k^{\mathfrak Z}D+\Pi_k^{\mathfrak H}D\).  Since
\(\zeta_{k+1}(R_k^ND)=\Pi_{k+1}^{\mathfrak Z}R_k^ND\), we have
\[
  \zeta_{k+1}(R_k^ND)-R_k^N\zeta_k(D)
  =
  \Pi_{k+1}^{\mathfrak Z}R_k^N\Pi_k^{\mathfrak H}D
  +
  \bigl(\Pi_{k+1}^{\mathfrak Z}R_k^N\Pi_k^{\mathfrak Z}D
  -
  R_k^N\Pi_k^{\mathfrak Z}D\bigr).
\]
The second term vanishes because \(R_k^N\Pi_k^{\mathfrak Z}D\in
\mathfrak Z_{k+1}\) and \(\Pi_{k+1}^{\mathfrak Z}\) is the identity on
\(\mathfrak Z_{k+1}\).  Thus
\[
  \Delta_{k}^{\sync,N}(D)
  =
  \norm{\Pi_{k+1}^{\mathfrak Z}R_k^N\Pi_k^{\mathfrak H}D}_{\mathfrak X_{k+1}},
\]
The stated bound follows from
\[
  \Pi_k^{\mathfrak H}D=D-\zeta_k(D).
\]
If \(R_k^N\mathfrak H_k\subset\mathfrak H_{k+1}\), then
\[
  \Pi_{k+1}^{\mathfrak Z}R_k^N\Pi_k^{\mathfrak H}D=0,
\]
so the drift vanishes.
\end{proof}

\begin{corollary}[Synchronization increment insertion]\label{cor:sync-insertion}
In a reduced finite-dimensional chain satisfying the hypotheses of
\cref{thm:sync-drift}, the synchronization component of the one-step ledger
may be chosen as
\[
  \Delta_k^{\sync}
  \ge
  \Delta_k^{\sync,N}(D).
\]
If the reduced map preserves both gauge and complement subspaces, then the
model synchronization contribution is zero.
\end{corollary}

\begin{proof}
The first assertion is the definition of assigning the computed drift to the
ledger component \(\Delta_k^{\sync}\).  The second assertion follows from
the last part of \cref{thm:sync-drift}.
\end{proof}

\subsection{Gate and slack update model}

The gate/slack component of the one-step ledger is another channel where a
finite-window algebraic update can be written without new PDE estimates.  Let
\(\mathfrak A\) be the finite set of active gate channels.  For
\(a\in\mathfrak A\), write \(B_{a,k}\ge0\) for the used budget,
\(\tau_{a,k}\ge0\) for the threshold, and \(s_{a,k}\ge0\) for the slack.
The gate violation and slack identity mismatch are
\[
  V_{a,k}:=(B_{a,k}-\tau_{a,k})_+,
  \qquad
  \Sigma_{a,k}:=\abs{B_{a,k}+s_{a,k}-\tau_{a,k}}.
\]

\begin{theorem}[Gate/slack one-step update]\label{thm:gate-update}
Assume that for each active gate channel \(a\in\mathfrak A\), the
renormalized budgets and thresholds satisfy
\[
  B_{a,k+1}
  \le
  \alpha_{a,k}B_{a,k}+\beta_{a,k}\mathfrak q_k+r_{a,k},
\]
and
\[
  \tau_{a,k+1}
  \ge
  \alpha_{a,k}\tau_{a,k}-q_{a,k},
\]
Here \(\mathfrak q_k\ge0\) is the non-gate disturbance supplied by the other
one-step channels.  The constants
\(\alpha_{a,k},\beta_{a,k},r_{a,k},q_{a,k}\) are nonnegative.  Then
\[
  V_{a,k+1}
  \le
  \alpha_{a,k}V_{a,k}
  +\beta_{a,k}\mathfrak q_k+r_{a,k}+q_{a,k}.
\]
If the next slack is chosen canonically by
\[
  s_{a,k+1}^{\mathrm{can}}:=(\tau_{a,k+1}-B_{a,k+1})_+,
\]
then
\[
  \abs{B_{a,k+1}+s_{a,k+1}^{\mathrm{can}}-\tau_{a,k+1}}
  =
  V_{a,k+1}.
\]
If exact slack identity is desired instead, the signed canonical choice
\[
  s_{a,k+1}^{\mathrm{sgn}}:=\tau_{a,k+1}-B_{a,k+1}
\]
gives
\[
  B_{a,k+1}+s_{a,k+1}^{\mathrm{sgn}}-\tau_{a,k+1}=0,
\]
but may fail to be nonnegative when the gate is violated.
\end{theorem}

\begin{proof}
Subtract the threshold lower bound from the budget upper bound:
\[
  B_{a,k+1}-\tau_{a,k+1}
  \le
  \alpha_{a,k}(B_{a,k}-\tau_{a,k})
  +\beta_{a,k}\mathfrak q_k+r_{a,k}+q_{a,k}.
\]
Taking positive parts and using monotonicity and subadditivity of
\((\cdot)_+\) gives
\[
  (B_{a,k+1}-\tau_{a,k+1})_+
  \le
  \alpha_{a,k}(B_{a,k}-\tau_{a,k})_+
  +\beta_{a,k}\mathfrak q_k+r_{a,k}+q_{a,k}.
\]
This is the claimed violation update.  For the nonnegative canonical slack,
if \(\tau_{a,k+1}\ge B_{a,k+1}\), then the displayed mismatch is zero; if
\(B_{a,k+1}>\tau_{a,k+1}\), then the canonical slack is zero and the
mismatch is exactly \(B_{a,k+1}-\tau_{a,k+1}=V_{a,k+1}\).  The signed slack
identity is immediate from the definition.
\end{proof}

\begin{corollary}[Gate increment insertion]\label{cor:gate-insertion}
Let
\[
  \Delta_k^{\gate}
  \ge
  \sum_{a\in\mathfrak A}
  \left(
    \beta_{a,k}\mathfrak q_k+r_{a,k}+q_{a,k}
  \right)
  +
  \sum_{a\in\mathfrak A}\alpha_{a,k}V_{a,k}.
\]
Then all next-scale gate violations generated by
\cref{thm:gate-update} are charged to the gate component of the one-step
ledger.  If a transported noncanonical slack \(\widetilde s_{a,k+1}\) is
used, add
\[
  \sum_{a\in\mathfrak A}
  \abs{\widetilde s_{a,k+1}-s_{a,k+1}^{\mathrm{can}}}
\]
to \(\Delta_k^{\gate}\).
\end{corollary}

\begin{proof}
Summing the estimate in \cref{thm:gate-update} over the finite channel set
\(\mathfrak A\) gives the first assertion.  For a transported slack,
\[
  \abs{B+\widetilde s-\tau}
  \le
  \abs{B+s^{\mathrm{can}}-\tau}
  +
  \abs{\widetilde s-s^{\mathrm{can}}},
\]
so the additional displayed term controls the noncanonical slack mismatch.
\end{proof}

\subsection{Reduced detector stability}

The detector update in \eqref{eq:audit-M-stability} is model-dependent in the full
package setting.  In a reduced finite-dimensional model it follows from a
standard commutator estimate between the detector map and the one-step
renormalization map.

\begin{definition}[Detector commutation defect]\label{def:det-comm}
Let \(\mathfrak X_k\) be a reduced package space and let
\[
  T_k:\mathfrak X_k\to\mathcal Y_k
\]
be the combined detector/tax-channel map into a finite-dimensional normed
space \(\mathcal Y_k\).  Define
\[
  M_k(D):=\norm{T_kD}_{\mathcal Y_k}.
\]
Let \(R_k^N:\mathfrak X_k\to\mathfrak X_{k+1}\) be the reduced one-step map
and let \(L_k:\mathcal Y_k\to\mathcal Y_{k+1}\) transport detector
coordinates.  The detector commutation defect is
\[
  \mathfrak C_k(D)
  :=
  T_{k+1}R_k^ND-L_kT_kD.
\]
\end{definition}

\begin{theorem}[Reduced detector-stability estimate]\label{thm:det-stability}
In the reduced detector model of \cref{def:det-comm},
\[
  M_{k+1}(R_k^ND)
  \le
  \norm{L_k}_{\mathcal Y_k\to\mathcal Y_{k+1}}M_k(D)
  +
  \norm{\mathfrak C_k(D)}_{\mathcal Y_{k+1}}.
\]
Consequently, if
\[
  \norm{\mathfrak C_k(D)}_{\mathcal Y_{k+1}}
  \le
  B_kE_k+C_k\Delta_k+\epsilon_k^{\detc},
\]
then
\[
  M_{k+1}(R_k^ND)
  \le
  A_kM_k(D)+B_kE_k+C_k\Delta_k+\epsilon_k^{\detc},
\]
with
\[
  A_k:=\norm{L_k}_{\mathcal Y_k\to\mathcal Y_{k+1}}.
\]
\end{theorem}

\begin{proof}
By the definition of \(\mathfrak C_k(D)\),
\[
  T_{k+1}R_k^ND=L_kT_kD+\mathfrak C_k(D).
\]
Taking the \(\mathcal Y_{k+1}\)-norm and using the triangle inequality gives
\[
  M_{k+1}(R_k^ND)
  \le
  \norm{L_kT_kD}_{\mathcal Y_{k+1}}
  +
  \norm{\mathfrak C_k(D)}_{\mathcal Y_{k+1}}.
\]
The first term is bounded by
\[
  \norm{L_k}_{\mathcal Y_k\to\mathcal Y_{k+1}}
  \norm{T_kD}_{\mathcal Y_k}
  =
  A_kM_k(D).
\]
The asserted detector update follows after inserting the assumed bound on
\(\mathfrak C_k(D)\).
\end{proof}

\begin{corollary}[Detector increment insertion]\label{cor:det-insertion}
In the reduced detector model, the detector component of the one-step ledger
may be chosen so that
\[
  \Delta_k^{\detc}
  \ge
  \norm{\mathfrak C_k(D)}_{\mathcal Y_{k+1}}.
\]
If the commutation defect satisfies the bound in
\cref{thm:det-stability}, then the detector update rule \eqref{eq:audit-M-stability}
holds with an additional additive tolerance \(\epsilon_k^{\detc}\).
\end{corollary}

\begin{proof}
The first assertion is the assignment of the detector commutation defect to
the detector ledger.  The second assertion is exactly
\cref{thm:det-stability}.
\end{proof}

\subsection{NS-generated pressure/source coordinate preservation}

\begin{theorem}[Pressure/source preservation]\label{thm:ns-pressure-source-preservation}
Let \((u,p)\) be pressure-admissible Navier--Stokes data on
\[
  Q_r(z_0)=B_r(x_0)\times(t_0-r^2,t_0)
\]
with
\[
  u\in L^3(Q_r(z_0))^3,\qquad p\in L^{3/2}(Q_r(z_0)),
\]
and
\[
  -\Delta p=\partial_i\partial_j(u_i u_j)
\]
in distributions, modulo time-dependent constants.  Define the normalized
fields on \(Q_1\) by
\[
  u_r(y,s)=r u(x_0+ry,t_0+r^2s),\qquad
  p_r(y,s)=r^2 p(x_0+ry,t_0+r^2s).
\]
Let \(\eta\in C_c^\infty(B_1)\) satisfy \(\eta\equiv1\) on \(B_{3/4}\), and
set
\[
  F_{r,ij}^{\mathrm{act}}:=\eta u_{r,i}u_{r,j},
  \qquad
  p_r^{\mathrm{act}}:=R_iR_j(F_{r,ij}^{\mathrm{act}}),
  \qquad
  p_r^{\harm}:=p_r-p_r^{\mathrm{act}}.
\]
Then
\[
  u_r\in X_u,\qquad
  F_r^{\mathrm{act}}\in X_{\mathrm{src}},\qquad
  p_r^{\mathrm{act}}\in Y_{\prs},\qquad
  p_r^{\harm}\in Y_{\harm}.
\]
Moreover, \(p_r^{\harm}\) is harmonic on \(B_{3/4}\) for almost every time.
\end{theorem}

\begin{proof}
The change of variables gives
\[
  \norm{u_r}_{L^3(Q_1)}^3
  =
  r^{-2}\norm{u}_{L^3(Q_r(z_0))}^3,
  \qquad
  \norm{p_r}_{L^{3/2}(Q_1)}^{3/2}
  =
  r^{-2}\norm{p}_{L^{3/2}(Q_r(z_0))}^{3/2},
\]
so \(u_r\in L^3(Q_1)^3\) and \(p_r\in L^{3/2}(Q_1)\).  Hence
\[
  F_r^{\mathrm{act}}=\eta u_r\otimes u_r\in X_{\mathrm{src}},
  \qquad
  \norm{F_r^{\mathrm{act}}}_{X_{\mathrm{src}}}
  \le
  \norm{u_r}_{L^3(Q_1)}^2.
\]
By the fixed-window Calderon--Zygmund estimate,
\[
  \norm{p_r^{\mathrm{act}}}_{Y_{\prs}}
  \le
  C_{\mathrm{CZ}}\norm{F_r^{\mathrm{act}}}_{X_{\mathrm{src}}}.
\]
The same global Calderon--Zygmund bound gives
\[
  \norm{p_r^{\mathrm{act}}}_{L^{3/2}((-1,0);L^{3/2}(B_{3/4}))}
  \le
  C_{\mathrm{CZ}}\norm{F_r^{\mathrm{act}}}_{X_{\mathrm{src}}}.
\]
The pressure-admissibility identity is invariant under Navier--Stokes
scaling:
\[
  -\Delta_y p_r
  =
  \partial_{y_i}\partial_{y_j}(u_{r,i}u_{r,j})
\]
in distributions.  Also,
\[
  -\Delta_y p_r^{\mathrm{act}}
  =
  \partial_{y_i}\partial_{y_j}(\eta u_{r,i}u_{r,j})
\]
in the same distributional convention for the double Riesz transform.  Since
\(\eta=1\) on \(B_{3/4}\), the difference \(p_r^{\harm}=p_r-p_r^{\mathrm{act}}\)
satisfies
\[
  -\Delta_y p_r^{\harm}=0
\]
on \(B_{3/4}\) for almost every time.  Finally,
\[
  \norm{p_r^{\harm}}_{Y_{\harm}}
  \le
  \norm{p_r}_{L^{3/2}((-1,0);L^{3/2}(B_{3/4}))}
  +
  \norm{p_r^{\mathrm{act}}}_{L^{3/2}((-1,0);L^{3/2}(B_{3/4}))},
\]
and the right-hand side is finite.
\end{proof}

\begin{corollary}[Pressure/source part of one-step admissibility]\label{cor:pressure-source-one-step}
For NS-generated packages obtained by restricting to \(Q_{k+1}\) and
renormalizing to \(Q_1\), use the canonical selected-source convention
\[
  S_{k+1}:=F_{r_{k+1}}^{\mathrm{act}}.
\]
Then the velocity, selected source, active pressure, and harmonic pressure
coordinates satisfy the coordinate membership requirements in
\cref{def:coord-check}.  Thus the pressure/source part of the one-step
admissibility checklist is verified.  The remaining synchronization,
localization, gate/slack, chart, clean-gap, and detector channels remain
explicit components of \(\Delta_k\).
\end{corollary}

\begin{proof}
Apply \cref{thm:ns-pressure-source-preservation} with
\(r=r_{k+1}\).  The final sentence records that the theorem does not estimate
the non-pressure/source channels; those are precisely the remaining
components in \eqref{eq:delta-k}.
\end{proof}

\subsection{Energy/flux localization update under renormalization}

This subsection fixes the localization convention used in the recursive
ledger.  It records the energy/flux leakage created by one restriction and
Navier--Stokes rescaling step.  The statement is deliberately a bookkeeping
estimate: it assigns the next-scale localization coordinate to the one-step
ledger, but it does not assert smallness or summability.

Let
\[
  u_{k+1}(Y,S)=\lambda u_k(\lambda Y,\lambda^2S),
  \qquad
  p_{k+1}(Y,S)=\lambda^2p_k(\lambda Y,\lambda^2S),
\]
where \(0<\lambda<1\).  Fix
\[
  0\le\chi\in C_c^\infty(B_1),
  \qquad
  \chi\equiv1\text{ on }B_{3/4},
\]
and \(0\le\theta\in C_c^\infty((-1,0])\).  Set
\[
  \phi(Y,S):=\theta(S)\chi(Y),
  \qquad
  A_\chi:=\supp\nabla\chi\cup\supp\Delta\chi,
  \qquad
  A_{\lambda,\chi}:=\lambda A_\chi.
\]

\begin{definition}[Energy/flux localization leakage]\label{def:ef-leak}
For finite-energy pressure-admissible data on \(Q_1\), define
\begin{align*}
  \Leak_\phi^{\mathrm{EF}}(u,p)
  :=
  &\int_{Q_1}|u|^2
  \left(\abs{\partial_t\phi}+\abs{\Delta\phi}\right)\,dx\,dt\\
  &+
  \int_{Q_1}\left(|u|^3+2|p||u|\right)
  \abs{\nabla\phi}\,dx\,dt
  +
  2\int_{Q_1}|\nabla u|^2\phi\,dx\,dt .
\end{align*}
The spatial-shell part is
\[
  \Leak_{\mathrm{shell}}^{\mathrm{EF}}(u,p)
  :=
  \int_{-1}^{0}\int_{A_\chi}
  \left(
    |u|^2+|\nabla u|^2+|u|^3+|p||u|
  \right)\,dx\,dt .
\]
The full pulled-back one-step energy/flux leakage is
\[
  \Leak_{k\to k+1}^{\mathrm{EF}}
  :=
  \int_{-\lambda^2}^{0}\int_{\lambda B_1}
  \left(|u_k|^2+|\nabla u_k|^2\right)\,dy\,ds
  +
  \int_{-\lambda^2}^{0}\int_{\lambda A_\chi}
  \left(|u_k|^2+|u_k|^3+|p_k||u_k|\right)\,dy\,ds .
\]
Its shell contribution is
\[
  \Leak_{k\to k+1}^{\mathrm{EF,shell}}
  :=
  \int_{-\lambda^2}^{0}\int_{\lambda A_\chi}
  \left(
    |u_k|^2+|\nabla u_k|^2+|u_k|^3+|p_k||u_k|
  \right)\,dy\,ds .
\]
\end{definition}

\begin{remark}
The full pulled-back leakage includes the core \(L^2\) and dissipation
contributions because the terms
\(\partial_t\phi\) and \(\phi|\nabla u|^2\) are not supported only on the
spatial transition shell.  The shell functional records the part generated
by the spatial cutoff.  Using the shell-only quantity for the full local
energy expression requires an additional convention that the time-cutoff and
core dissipation terms have already been charged to a separate energy
budget.
\end{remark}

\begin{theorem}[Energy/flux localization update under one-step renormalization]\label{thm:ef-loc-update}
Let \(D_k\) be an NS-generated scale-\(k\) package generated by data
satisfying, in normalized scale-\(k\) variables,
\[
  u_k\in L^\infty((-1,0);L^2(B_1))^3,\qquad
  \nabla u_k\in L^2(Q_1)^{3\times3},
\]
\[
  u_k\in L^3(Q_1)^3,\qquad
  p_k\in L^{3/2}(Q_1).
\]
Let \(D_{k+1}=\mathcal R_k(D_k)\) be obtained by restriction to
\((-\lambda^2,0)\times B_\lambda\), Navier--Stokes rescaling to \(Q_1\),
and recomputation of the next-scale package.  Then
\[
  \Leak_\phi^{\mathrm{EF}}(u_{k+1},p_{k+1})
  \le
  C_{\lambda,\chi,\theta}\,
  \Leak_{k\to k+1}^{\mathrm{EF}} .
\]
Consequently, if
\[
  \Delta_k^{\locerr,\mathrm{EF}}
  \ge
  C_{\lambda,\chi,\theta}\,
  \Leak_{k\to k+1}^{\mathrm{EF}},
\]
then the energy/flux localization component of the scale-\((k+1)\)
coordinate admissibility checklist is charged to the one-step ledger.
\end{theorem}

\begin{proof}
Use the change of variables
\[
  y=\lambda Y,\qquad s=\lambda^2S,\qquad dY\,dS=\lambda^{-5}\,dy\,ds.
\]
The rescaled fields satisfy
\[
  |u_{k+1}(Y,S)|=\lambda |u_k(y,s)|,\qquad
  |p_{k+1}(Y,S)|=\lambda^2 |p_k(y,s)|,
\]
and
\[
  |\nabla_Yu_{k+1}(Y,S)|
  =
  \lambda^2|\nabla_yu_k(y,s)|.
\]
The constants
\[
  \|\theta\|_{C^1},\qquad
  \|\chi\|_{C^2}
\]
control \(\abs{\partial_S\phi}\), \(\abs{\nabla_Y\phi}\), and
\(\abs{\Delta_Y\phi}\).  The terms containing \(\nabla_Y\chi\) or
\(\Delta_Y\chi\) are supported in \((-\lambda^2,0)\times\lambda A_\chi\)
after pullback.  The time-cutoff energy term and the dissipation term are
supported in \((-\lambda^2,0)\times\lambda B_1\).

Therefore
\[
  \int_{Q_1}|u_{k+1}|^2
  \left(\abs{\partial_S\phi}+\abs{\Delta_Y\phi}\right)\,dY\,dS
  \le
  C_{\lambda,\chi,\theta}
  \int_{-\lambda^2}^{0}\int_{\lambda B_1}|u_k|^2\,dy\,ds,
\]
\[
  \int_{Q_1}\left(|u_{k+1}|^3+2|p_{k+1}||u_{k+1}|\right)
  \abs{\nabla_Y\phi}\,dY\,dS
  \le
  C_{\lambda,\chi,\theta}
  \int_{-\lambda^2}^{0}\int_{\lambda A_\chi}
  \left(|u_k|^3+|p_k||u_k|\right)\,dy\,ds,
\]
and
\[
  2\int_{Q_1}|\nabla_Yu_{k+1}|^2\phi\,dY\,dS
  \le
  C_{\lambda,\chi,\theta}
  \int_{-\lambda^2}^{0}\int_{\lambda B_1}|\nabla_yu_k|^2\,dy\,ds.
\]
The constants absorb the fixed powers of \(\lambda\), including
\(\lambda^{-3}\), \(\lambda^{-2}\), and \(\lambda^{-1}\).  Adding the three
estimates gives the claimed bound.  The final assertion is the definition of
the localization component of the ledger.
\end{proof}

\begin{corollary}[Finite-amplitude localization bookkeeping]\label{cor:ef-finite-amplitude}
If \(\norm{u_k}_{L^3(Q_1)}\le M_U\), then
\[
  \int_{-\lambda^2}^{0}\int_{\lambda A_\chi}|p_k||u_k|\,dy\,ds
  \le
  \norm{p_k}_{L^{3/2}((-\lambda^2,0)\times\lambda A_\chi)}
  \norm{u_k}_{L^3((-\lambda^2,0)\times\lambda A_\chi)}
\]
and
\[
  \int_{-\lambda^2}^{0}\int_{\lambda A_\chi}|u_k|^3\,dy\,ds
  =
  \norm{u_k}_{L^3((-\lambda^2,0)\times\lambda A_\chi)}^3 .
\]
Thus the nonlinear flux and pressure-flux parts of
\(\Leak_{k\to k+1}^{\mathrm{EF}}\) are charged to the existing finite-window
velocity and pressure budgets.  No smallness is implied without additional
smallness assumptions.
\end{corollary}

\begin{proof}
The pressure-flux estimate is Holder's inequality with exponents \(3/2\)
and \(3\).  The cubic identity is the definition of the \(L^3\)-norm.  The
finite-amplitude hypothesis records that the global velocity factor is
finite when such a factor is needed elsewhere in the ledger; it does not
turn the displayed quantities into small errors.
\end{proof}

\begin{corollary}[Localization increment insertion]\label{cor:ef-loc-insertion}
In the energy/flux convention, set
\[
  e_{\locerr,k+1}^{\mathrm{EF}}
  :=
  C_{\lambda,\chi,\theta}\Leak_{k\to k+1}^{\mathrm{EF}}.
\]
If
\[
  \Delta_k^{\locerr}
  =
  \Delta_k^{\locerr,\mathrm{EF}}
  \ge
  e_{\locerr,k+1}^{\mathrm{EF}},
\]
then the localization entry in the scale-\((k+1)\) checklist is verified.
Consequently, if the pressure/source coordinates are preserved and the
synchronization, gate/slack, detector, chart, clean-gap, projection, harmonic,
and energy/flux localization entries are bounded by their assigned ledger
terms, then
\[
  D_{k+1}\in\mathcal A_{k+1}(\Delta_k).
\]
\end{corollary}

\begin{proof}
The first assertion is exactly \cref{thm:ef-loc-update} with the
localization entry named as a checklist excess.  The final assertion is
\cref{thm:coord-adm} applied with this localization entry and the other
entries assigned to the remaining components of \(\Delta_k\).
\end{proof}

\begin{remark}
\Cref{thm:ef-loc-update} is a finite-window localization bookkeeping result.
It proves that the localization leakage generated by one restriction and
rescaling step can be assigned to the one-step ledger.  It does not prove
that the leakage is small, summable, or scale-uniform.
\end{remark}

\subsection{Coefficient update and finite-window positivity}

The one-step audit theorem uses the coefficient update
\[
  c_{k+1}\ge c_k-\eta_k.
\]
At the finite-window level this update is an algebraic consequence of how the
clean gap, chart visibility, and residual-loss coefficient change from one
renormalized window to the next.

\begin{definition}[Scale coefficient drift]\label{def:coeff-drift}
Write
\[
  \mu_k:=\mu_{\Lambda,k}^{\comp},
  \qquad
  \lambda_k:=\lambda_{G,k},
  \qquad
  \ell_k:=\mathcal L_k^{\mathrm{res}},
\]
and
\[
  c_k:=\mu_k\lambda_k-\ell_k.
\]
The coefficient drift from scale \(k\) to \(k+1\) is controlled by
nonnegative numbers
\[
  \delta_{\mu,k},\qquad
  \delta_{\lambda,k},\qquad
  \delta_{\ell,k},
\]
if
\[
  \mu_{k+1}\ge(\mu_k-\delta_{\mu,k})_+,
  \qquad
  \lambda_{k+1}\ge(\lambda_k-\delta_{\lambda,k})_+,
  \qquad
  \ell_{k+1}\le \ell_k+\delta_{\ell,k}.
\]
Define the induced coefficient-loss allowance
\[
  \eta_k^{\mathrm{coeff}}
  :=
  \mu_k\lambda_k
  -
  (\mu_k-\delta_{\mu,k})_+
  (\lambda_k-\delta_{\lambda,k})_+
  +
  \delta_{\ell,k}.
\]
\end{definition}

\begin{theorem}[Finite-window coefficient update]\label{thm:coeff-update}
Assume the coefficient drift bounds in \cref{def:coeff-drift}.  Then
\[
  c_{k+1}\ge c_k-\eta_k^{\mathrm{coeff}}.
\]
Consequently, the abstract update rule \eqref{eq:audit-c-update} holds whenever
\[
  \eta_k\ge\eta_k^{\mathrm{coeff}}.
\]
\end{theorem}

\begin{proof}
By the drift assumptions,
\[
  c_{k+1}
  =
  \mu_{k+1}\lambda_{k+1}-\ell_{k+1}
  \ge
  (\mu_k-\delta_{\mu,k})_+
  (\lambda_k-\delta_{\lambda,k})_+
  -
  \ell_k-\delta_{\ell,k}.
\]
Since \(c_k=\mu_k\lambda_k-\ell_k\), the right-hand side equals
\[
  c_k
  -
  \left[
    \mu_k\lambda_k
    -
    (\mu_k-\delta_{\mu,k})_+
    (\lambda_k-\delta_{\lambda,k})_+
    +
    \delta_{\ell,k}
  \right]
  =
  c_k-\eta_k^{\mathrm{coeff}}.
\]
This proves the estimate.  The final assertion follows by choosing the
update-rule loss \(\eta_k\) at least as large as the displayed coefficient
loss.
\end{proof}

\begin{corollary}[Linear coefficient-loss bound]\label{cor:linear-coeff-loss}
If
\[
  0\le\delta_{\mu,k}\le\mu_k,
  \qquad
  0\le\delta_{\lambda,k}\le\lambda_k,
\]
then
\[
  \eta_k^{\mathrm{coeff}}
  \le
  \mu_k\delta_{\lambda,k}
  +
  \lambda_k\delta_{\mu,k}
  +
  \delta_{\ell,k}.
\]
If, in addition, \(\mu_k\le\overline\mu_k\) and
\(\lambda_k\le\overline\lambda_k\), then the simpler bound
\[
  \eta_k^{\mathrm{coeff}}
  \le
  \overline\mu_k\delta_{\lambda,k}
  +
  \overline\lambda_k\delta_{\mu,k}
  +
  \delta_{\ell,k}
\]
also holds.
\end{corollary}

\begin{proof}
Under the displayed size assumptions,
\[
  (\mu_k-\delta_{\mu,k})_+
  (\lambda_k-\delta_{\lambda,k})_+
  =
  (\mu_k-\delta_{\mu,k})(\lambda_k-\delta_{\lambda,k}).
\]
Therefore
\[
  \eta_k^{\mathrm{coeff}}
  =
  \mu_k\delta_{\lambda,k}
  +
  \lambda_k\delta_{\mu,k}
  -
  \delta_{\mu,k}\delta_{\lambda,k}
  +
  \delta_{\ell,k}
  \le
  \mu_k\delta_{\lambda,k}
  +
  \lambda_k\delta_{\mu,k}
  +
  \delta_{\ell,k}.
\]
The upper-bound formulation follows immediately.
\end{proof}

\begin{corollary}[Finite-chain positivity persistence]\label{cor:coeff-positivity}
If
\[
  c_0>\sum_{j=0}^{K-1}\eta_j^{\mathrm{coeff}},
\]
then
\[
  c_k
  \ge
  c_0-\sum_{j=0}^{k-1}\eta_j^{\mathrm{coeff}}>0
  \qquad\text{for }0\le k\le K.
\]
In particular,
\[
  c_K^{\min}
  :=
  \min_{0\le k\le K}c_k
  \ge
  c_0-\sum_{j=0}^{K-1}\eta_j^{\mathrm{coeff}}.
\]
\end{corollary}

\begin{proof}
Iterating \cref{thm:coeff-update} gives
\[
  c_k\ge c_0-\sum_{j=0}^{k-1}\eta_j^{\mathrm{coeff}}.
\]
The assumed strict inequality makes the right-hand side positive for every
\(k\le K\).  Taking the minimum over \(k\) gives the displayed lower bound for
\(c_K^{\min}\).
\end{proof}

\begin{remark}
\Cref{thm:coeff-update} is purely finite-window algebra.  It does not prove
that \(\mu_k\), \(\lambda_k\), or \(\ell_k\) are scale-uniform.  It only
identifies the exact coefficient loss that must be included in the recursive
ledger when these quantities drift from one scale to the next.
\end{remark}

\section{Broad One-Step Admissibility Theorem}
\label{sec:main-results}

\subsection{Working one-step theorem}

\begin{theorem}[NS-generated one-step admissibility under explicit ledger assignment]
\label{thm:main}
Let \(D_k\in\A_k^{\NS}\) be generated by pressure-admissible finite-energy
Navier--Stokes data on the normalized scale-\(k\) window.  Let
\[
  D_{k+1}=\mathcal R_kD_k
\]
be obtained by restriction, Navier--Stokes rescaling, synchronization, and
recomputation of the scale-\((k+1)\) package.  Assume:
\begin{enumerate}[label=(\alph*)]
  \item pressure/source preservation holds as in
  \cref{prop:pressure-source-preservation};
  \item the energy/flux localization component is charged as in
  \cref{prop:ef-localization};
  \item pressure recomputation is charged by the commutator estimate in
  \cref{prop:commutator};
  \item synchronization uses the broad near-minimizer synchronization loss
  convention in \cref{def:broad-sync}, with
  \(\Delta_k^{\sync}\ge\mathrm{SyncErr}_k\), so that
  \cref{prop:broad-sync-insertion} applies;
  \item projection and harmonic tails use the finite-window datum in
  \cref{def:proj-harm-datum}, with
  \(\Delta_k^{\proj}\ge\mathrm{ProjErr}_k\) and
  \(\Delta_k^{\harm}\ge\mathrm{HarmErr}_k\), so that
  \cref{prop:proj-harm-insertion} applies;
  \item gate/slack uses the finite-window datum in
  \cref{def:gate-slack-datum}, with
  \(\Delta_k^{\gate}\ge\mathrm{GateErr}_k\), so that
  \cref{prop:gate-slack-insertion} applies;
  \item the detector channel is realized by the reduced PFE detector
  comparison in \cref{thm:reduced-pfe-detector-comparison}, with
  \(\Delta_k^{\detc}\ge\mathrm{DetErr}_k^{\mathrm{PFE}}\), so that
  \cref{cor:reduced-pfe-detector-ledger} applies;
  \item chart and clean-gap mismatches are assigned as in
  \cref{ass:reduced-detector}.
\end{enumerate}
Then
\[
  D_{k+1}\in\A_{k+1}^{\NS}(\Delta_k),
\]
where
\[
  \Delta_k
  =
  \Delta_k^{\sync}
  +\Delta_k^{\mathrm{loc},\EF}
  +\Delta_k^{\proj}
  +\Delta_k^{\harm}
  +\Delta_k^{\chart}
  +\Delta_k^{\gap}
  +\Delta_k^{\gate}
  +\Delta_k^{\detc}.
\]
\end{theorem}

\begin{remark}[Status of the theorem]
This is a conditional finite-window theorem.  The pressure/source,
energy/flux localization, and commutator modules in \cref{sec:pde}, together
with the reduced PFE detector comparison in \cref{sec:structural}, are proved
at fixed finite-window level under their stated hypotheses.  The chart and
clean-gap assignments in \cref{sec:structural} remain structural inputs unless
separately verified.
\end{remark}

\begin{proof}
The proof is an assembly of the preceding modules.  Applied to the rescaled
next-scale data, the pressure/source proposition verifies the active pressure,
source, and harmonic pressure coordinate entries at scale \(k+1\).  The local energy inequality and
rescaling estimate assign all localization terms to
\(\Delta_k^{\mathrm{loc},\EF}\).  The commutator estimate assigns the pressure
recomputation loss.  Broad synchronization is inserted by
\cref{prop:broad-sync-insertion}.  The projection and harmonic tails are
inserted by \cref{prop:proj-harm-insertion}.  The gate/slack entry is inserted
by \cref{prop:gate-slack-insertion}.  The detector channel is inserted by the
reduced PFE detector comparison in \cref{cor:reduced-pfe-detector-ledger}.
The remaining chart and clean-gap terms are inserted through
\cref{ass:reduced-detector}.  With every defining mismatch of
\(\A_{k+1}^{\NS}\) dominated by the corresponding component of \(\Delta_k\),
the package belongs to \(\A_{k+1}^{\NS}(\Delta_k)\).
\end{proof}

\subsection{Insertion into the recursive finite-chain theorem}

\begin{corollary}[One-step input for recursive audit chains]
\label{cor:recursive-insertion}
Assume the finite recursive audit theorem from the preceding paper applies to
any finite chain satisfying \(D_{k+1}\in\A_{k+1}^{\NS}(\Delta_k)\).  If the
hypotheses of \cref{thm:main} hold for each \(k=0,\ldots,K-1\), then the chain
is admissible for the finite recursive audit theorem in \cref{sec:finite-chain-recursive}.
\end{corollary}

\begin{remark}[Status]
\Cref{cor:recursive-insertion} is an insertion statement.  It does not prove
scale-uniformity or summability of \(\Delta_k\).
\end{remark}

\section{Static Audit Certificates and Recursive Update Rules}
\label{sec:static-recursive-rules}

The broad one-step theorem supplies admissible transitions.  The recursive audit theorem also needs a static finite-window certificate on each scale and an error/coefficient update rule.

\begin{definition}[Scale audit certificate]\label{def:scale-audit-certificate}
A scale-
\(k\) audit certificate is valid if
\begin{equation}\label{eq:scale-audit-cert}
  M_k\ge c_k\delta_k-E_k,
\end{equation}
where \(M_k\) is the localized detector value, \(\delta_k\) is the baseline quotient defect distance, \(E_k\) is the total residual/error budget, and \(c_k\ge0\) is the finite-window audit coefficient.
\end{definition}

\begin{assumption}[Recursive update rules]\label{ass:recursive-update-rules}
For each step \(k\), assume constants \(a_k,b_k,\eta_k\ge0\) such that
\begin{equation}\label{eq:recursive-E-update}
  E_{k+1}\le a_kE_k+b_k\Delta_k,
\end{equation}
and
\begin{equation}\label{eq:audit-c-update}
  c_{k+1}\ge c_k-\eta_k.
\end{equation}
When detector stability is needed as a separate upper-control channel, assume also constants \(A_k,B_k,C_k\ge0\) such that
\begin{equation}\label{eq:audit-M-stability}
  M_{k+1}\le A_kM_k+B_kE_k+C_k\Delta_k.
\end{equation}
\end{assumption}

\begin{theorem}[One-step audit propagation]\label{thm:one-step-audit-propagation}
Let \(D_k\in\A_k^{\NS}\), let \(D_{k+1}=R_kD_k\), and assume the one-step admissibility conclusion
\[
  D_{k+1}\in\A_{k+1}^{\NS}(\Delta_k).
\]
Assume that the scale-\((k+1)\) static finite-window audit theorem applies to \(D_{k+1}\), giving
\[
  M_{k+1}\ge c_{k+1}\delta_{k+1}-E_{k+1}.
\]
Then \(D_{k+1}\) carries a valid audit certificate.  Moreover the coefficient and residual budget obey the recursive controls \eqref{eq:recursive-E-update} and \eqref{eq:audit-c-update}.  Consequently, if the hypotheses hold for \(k=0,\ldots,K-1\), every package in the finite chain carries a valid audit certificate with explicitly propagated error and coefficient loss.
\end{theorem}

\begin{proof}
The scale-\((k+1)\) certificate is exactly the static finite-window audit theorem applied after admissibility has been reverified at the next scale.  The displayed update inequalities are the recursive bookkeeping assumptions.  Iterating the same argument over a finite number of steps proves the final statement.
\end{proof}

\section{Finite-Chain Recursive Anti-Phantom Theorem}
\label{sec:finite-chain-recursive}

\subsection{Error recursion algebra}

\begin{lemma}[Variable-coefficient error recursion]\label{lem:var-recursion}
Suppose
\[
  E_{k+1}\le a_kE_k+b_k\Delta_k,\qquad k=0,\dots,K-1.
\]
With the convention that an empty product equals \(1\), one has
\[
  E_K
  \le
  \left(\prod_{i=0}^{K-1}a_i\right)E_0
  +
  \sum_{j=0}^{K-1}
  b_j
  \left(\prod_{i=j+1}^{K-1}a_i\right)
  \Delta_j .
\]
\end{lemma}

\begin{proof}
For \(K=1\) the claim is the assumed recursion.  Assume it is true for \(K\).
Then
\[
  E_{K+1}\le a_KE_K+b_K\Delta_K.
\]
Substituting the induction hypothesis for \(E_K\) gives
\[
  E_{K+1}
  \le
  \left(\prod_{i=0}^{K}a_i\right)E_0
  +
  \sum_{j=0}^{K-1}
  b_j
  \left(\prod_{i=j+1}^{K}a_i\right)\Delta_j
  +b_K\Delta_K,
\]
which is the desired formula at \(K+1\).
\end{proof}

\begin{corollary}[Constant-coefficient recursion]\label{cor:const-recursion}
If \(E_{k+1}\le aE_k+b\Delta_k\), then
\[
  E_K\le a^K E_0+b\sum_{j=0}^{K-1}a^{K-1-j}\Delta_j.
\]
\end{corollary}

\subsection{Weighted chain lower bound}

\begin{theorem}[Finite-chain recursive anti-phantom theorem]\label{thm:finite-chain}
Let \(D_0,\dots,D_K\) be a finite renormalized audit chain and let
\(w_k\ge0\).  Assume each scale carries a valid certificate
\[
  M_k\ge c_k\delta_k-E_k.
\]
Set
\[
  c_K^{\min}:=\min_{0\le k\le K}c_k
\]
and
\[
  \mathcal E_K^{\mathrm{rec}}:=\sum_{k=0}^K w_kE_k.
\]
Then
\begin{equation}\label{eq:finite-chain-lower}
  \sum_{k=0}^K w_kM_k
  \ge
  c_K^{\min}\sum_{k=0}^K w_k\delta_k
  -
  \mathcal E_K^{\mathrm{rec}}.
\end{equation}
Moreover, the terms \(E_k\) are controlled by \cref{lem:var-recursion} on
each prefix of the chain.
\end{theorem}

\begin{proof}
Multiplying \(M_k\ge c_k\delta_k-E_k\) by \(w_k\ge0\) and summing gives
\[
  \sum_{k=0}^K w_kM_k
  \ge
  \sum_{k=0}^K w_kc_k\delta_k
  -
  \sum_{k=0}^K w_kE_k.
\]
Since \(c_k\ge c_K^{\min}\), this is \eqref{eq:finite-chain-lower}.  The
claimed control of \(E_k\) follows from \cref{lem:var-recursion} applied with
terminal index \(k\).
\end{proof}

\begin{corollary}[Recursive anti-phantom alternative]\label{cor:recursive-alt}
Under the hypotheses of \cref{thm:finite-chain}, at least one of the
following alternatives holds:
\[
  \sum_{k=0}^K w_kM_k
  \ge
  \frac{c_K^{\min}}{2}
  \sum_{k=0}^K w_k\delta_k,
\]
or
\[
  \mathcal E_K^{\mathrm{rec}}
  \ge
  \frac{c_K^{\min}}{2}
  \sum_{k=0}^K w_k\delta_k.
\]
\end{corollary}

\begin{proof}
If
\[
  \mathcal E_K^{\mathrm{rec}}
  \le
  \frac{c_K^{\min}}{2}\sum_{k=0}^K w_k\delta_k,
\]
then \eqref{eq:finite-chain-lower} gives the first alternative.  If this
inequality fails, the second alternative holds.
\end{proof}

\section{Recursive Error Regimes}

Consider the constant-coefficient recursion
\[
  E_{k+1}\le aE_k+b\Delta_k.
\]

\begin{proposition}[Contractive regime]\label{prop:contractive}
If \(0\le a<1\) and \(\Delta_k\le \Delta_*\) for all \(k\), then
\[
  E_K
  \le
  a^K E_0+\frac{b}{1-a}\Delta_*.
\]
\end{proposition}

\begin{proof}
Use \cref{cor:const-recursion} and the geometric-series bound
\[
  \sum_{j=0}^{K-1}a^{K-1-j}\le \frac{1}{1-a}.
\]
\end{proof}

\begin{proposition}[Neutral regime]\label{prop:neutral}
If \(a=1\), then
\[
  E_K\le E_0+b\sum_{j=0}^{K-1}\Delta_j.
\]
Thus finite-chain propagation remains useful whenever the accumulated
increments are controlled on the considered chain.
\end{proposition}

\begin{proof}
This is \cref{cor:const-recursion} with \(a=1\).
\end{proof}

\begin{proposition}[Expanding regime]\label{prop:expanding}
If \(a>1\), then
\[
  E_K
  \le
  a^K E_0+b\sum_{j=0}^{K-1}a^{K-1-j}\Delta_j.
\]
Finite \(K\) estimates still hold, but scale-uniform propagation requires
additional decay or cancellation in \(E_0\) and the increments
\(\Delta_j\).
\end{proposition}

\begin{proof}
This is again \cref{cor:const-recursion}.  The final statement is an
interpretation of the exponential weights \(a^{K-1-j}\).
\end{proof}

\section{Compact, Reduced, and PFE Structural Inputs}

This section verifies finite-window structural inputs for selected compact,
smooth, reduced, or effectively projected package classes.  No result in
this section is asserted for all suitable weak Navier--Stokes solutions.

\subsection{Pressure projection compactness and effective projection}

Let
\[
  X_{\mathrm{src}}
  :=
  L^{3/2}((-1,0);L^{3/2}(B_1))^{3\times3},
  \qquad
  Y_{\prs}
  :=
  L^{3/2}((-1,0);L^{3/2}(B_{1/2})).
\]
The clean pressure map is
\[
  \Rprs F:=R_iR_j(F_{ij})\big|_{B_{1/2}},
\]
with sources extended by zero outside \(B_1\).  We assume the fixed
Calderon--Zygmund estimate
\[
  \norm{\Rprs F}_{Y_{\prs}}\le C_{\mathrm{CZ}}\norm{F}_{X_{\mathrm{src}}}.
\]

\begin{definition}[Selected source and pressure image]\label{def:selected-source}
At scale \(k\), define
\[
  \F_{k,0}^{\NS}
  :=
  \{F_{D_k-\zeta_k}^{\cl}:D_k\in\A_k^{\NS}\}
  \subset X_{\mathrm{src}},
\]
and
\[
  \mathcal G_{k,0}^{\NS}
  :=
  \{\Rprs F:F\in\F_{k,0}^{\NS}\}
  \subset Y_{\prs}.
\]
For finite-rank clean pressure projections \(P_{N,k}^{\cl}:Y_{\prs}\to
Y_{\prs}\), set
\[
  \Delta_{k,\proj,N}^{\mathrm{unif}}
  :=
  \sup_{g\in\mathcal G_{k,0}^{\NS}}
  \norm{(I-P_{N,k}^{\cl})g}_{Y_{\prs}}.
\]
\end{definition}

\begin{theorem}[Compact pressure image gives uniform projection tails]\label{thm:compact-image-proj}
Assume \(\mathcal G_{k,0}^{\NS}\Subset Y_{\prs}\),
\[
  P_{N,k}^{\cl}g\to g\quad\text{in }Y_{\prs}\text{ for every }g\in Y_{\prs},
\]
and
\[
  C_P:=\sup_N\norm{P_{N,k}^{\cl}}_{Y_{\prs}\to Y_{\prs}}<\infty.
\]
Then
\[
  \Delta_{k,\proj,N}^{\mathrm{unif}}\to0.
\]
\end{theorem}

\begin{proof}
Let \(K_k=\overline{\mathcal G_{k,0}^{\NS}}\), compact in \(Y_{\prs}\).
Fix \(\varepsilon>0\).  Choose a finite
\(\varepsilon/(3(1+C_P))\)-net \(g_1,\dots,g_J\) for \(K_k\).  For each
center, strong convergence gives
\[
  \norm{(I-P_{N,k}^{\cl})g_j}_{Y_{\prs}}\to0.
\]
For \(N\) large, all these quantities are at most \(\varepsilon/3\).  Given
\(g\in K_k\), choose \(g_j\) with
\[
  \norm{g-g_j}_{Y_{\prs}}\le \frac{\varepsilon}{3(1+C_P)}.
\]
Then
\[
  \norm{(I-P_{N,k}^{\cl})g}_{Y_{\prs}}
  \le
  (1+C_P)\norm{g-g_j}_{Y_{\prs}}
  +\norm{(I-P_{N,k}^{\cl})g_j}_{Y_{\prs}}
  \le
  \frac{2\varepsilon}{3}.
\]
Taking the supremum over \(g\in\mathcal G_{k,0}^{\NS}\) proves the claim.
\end{proof}

\begin{theorem}[Source compactness implies pressure-image compactness]\label{thm:source-compact}
If \(\F_{k,0}^{\NS}\Subset X_{\mathrm{src}}\), then
\[
  \mathcal G_{k,0}^{\NS}\Subset Y_{\prs},
\]
and hence \(\Delta_{k,\proj,N}^{\mathrm{unif}}\to0\).
\end{theorem}

\begin{proof}
The map \(\Rprs:X_{\mathrm{src}}\to Y_{\prs}\) is bounded linear, hence
continuous.  Continuous images of compact sets are compact.  The last claim
follows from \cref{thm:compact-image-proj}.
\end{proof}

\begin{proposition}[Compactness criteria]\label{prop:compactness-criteria}
Each of the following hypotheses implies
\(\F_{k,0}^{\NS}\Subset X_{\mathrm{src}}\).
\begin{enumerate}[label=(\alph*)]
  \item Smooth finite-window class: the selected sources are generated by a
  family of smooth Navier--Stokes data with a uniform \(C^m\) or \(H^s\)
  bound strong enough to compactly embed into \(L^{3/2}\) on the fixed
  window.
  \item Sobolev compactness: for some \(s>0\),
  \[
    \sup_{F\in\F_{k,0}^{\NS}}\norm{F}_{W^{s,3/2}(Q_1)}<\infty.
  \]
  \item Kolmogorov--Riesz compactness: the family is uniformly bounded in
  \(L^{3/2}\), supported in the fixed window, and uniformly translation
  continuous in \(L^{3/2}\) after zero extension.
\end{enumerate}
\end{proposition}

\begin{proof}
Part (a) follows from Arzela--Ascoli or Rellich--Kondrachov, depending on
whether a \(C^m\) or \(H^s\) formulation is used.  Part (b) is the
Rellich--Kondrachov compact embedding
\[
  W^{s,3/2}(Q_1)\Subset L^{3/2}(Q_1)
\]
on a bounded domain, applied componentwise.  Part (c) is the
Kolmogorov--Riesz compactness theorem on a bounded domain after zero
extension.
\end{proof}

\begin{proposition}[Effective projection replacement]\label{prop:effective-proj}
For recursive audit estimates, compactness may be replaced by the direct
assumption
\[
  \Delta_{k,\proj,N}^{\mathrm{unif}}\le \varepsilon_{k,N},
  \qquad
  \varepsilon_{k,N}\to0.
\]
For a finite weighted chain,
\[
  \sum_{k=0}^K w_k\Delta_{k,\proj,N}^{\mathrm{unif}}
  \le
  \varepsilon_{K,N}^{\mathrm{chain}}
  :=
  \sum_{k=0}^K w_k\varepsilon_{k,N}.
\]
\end{proposition}

\begin{proof}
The statement is immediate from the definition of the effective projection
bound and summation with nonnegative weights.
\end{proof}

\begin{definition}[Projection-tail transport data]\label{def:proj-transport}
Let
\[
  T_k^{\prs}:Y_{\prs}\to Y_{\prs}
\]
be a bounded pressure-coordinate transport map from scale \(k\) to scale
\(k+1\).  The transported pressure-image defect is a number
\(\delta_{k}^{\mathrm{img}}\ge0\) such that for every
\[
  g_{k+1}\in\mathcal G_{k+1,0}^{\NS}
\]
there exists \(g_k\in\mathcal G_{k,0}^{\NS}\) satisfying
\[
  \norm{g_{k+1}-T_k^{\prs}g_k}_{Y_{\prs}}
  \le
  \delta_k^{\mathrm{img}}.
\]
For finite-rank projections \(P_{N,k}^{\cl}\) and \(P_{N,k+1}^{\cl}\), define
the projection-transport commutator size
\[
  \omega_{k,N}^{\proj}
  :=
  \sup_{g\in\mathcal G_{k,0}^{\NS}}
  \norm{
  (I-P_{N,k+1}^{\cl})T_k^{\prs}g
  -
  T_k^{\prs}(I-P_{N,k}^{\cl})g
  }_{Y_{\prs}}.
\]
\end{definition}

\begin{theorem}[Projection-tail drift under pressure-image transport]\label{thm:proj-tail-drift}
Assume the projection-tail transport data in \cref{def:proj-transport}.  Set
\[
  C_{P,k+1}:=\norm{P_{N,k+1}^{\cl}}_{Y_{\prs}\to Y_{\prs}},
  \qquad
  C_{T,k}:=\norm{T_k^{\prs}}_{Y_{\prs}\to Y_{\prs}}.
\]
Then
\[
  \Delta_{k+1,\proj,N}^{\mathrm{unif}}
  \le
  (1+C_{P,k+1})\delta_k^{\mathrm{img}}
  +
  C_{T,k}\Delta_{k,\proj,N}^{\mathrm{unif}}
  +
  \omega_{k,N}^{\proj}.
\]
\end{theorem}

\begin{proof}
Fix \(g_{k+1}\in\mathcal G_{k+1,0}^{\NS}\).  Choose
\(g_k\in\mathcal G_{k,0}^{\NS}\) with
\[
  \norm{g_{k+1}-T_k^{\prs}g_k}_{Y_{\prs}}
  \le
  \delta_k^{\mathrm{img}}.
\]
Then
\begin{align*}
  (I-P_{N,k+1}^{\cl})g_{k+1}
  &=
  (I-P_{N,k+1}^{\cl})(g_{k+1}-T_k^{\prs}g_k)\\
  &\quad
  +
  (I-P_{N,k+1}^{\cl})T_k^{\prs}g_k.
\end{align*}
For the first term,
\[
  \norm{(I-P_{N,k+1}^{\cl})(g_{k+1}-T_k^{\prs}g_k)}_{Y_{\prs}}
  \le
  (1+C_{P,k+1})\delta_k^{\mathrm{img}}.
\]
For the second term, add and subtract
\(T_k^{\prs}(I-P_{N,k}^{\cl})g_k\):
\begin{align*}
  \norm{(I-P_{N,k+1}^{\cl})T_k^{\prs}g_k}_{Y_{\prs}}
  &\le
  \norm{T_k^{\prs}(I-P_{N,k}^{\cl})g_k}_{Y_{\prs}}\\
  &\quad+
  \norm{
  (I-P_{N,k+1}^{\cl})T_k^{\prs}g_k
  -
  T_k^{\prs}(I-P_{N,k}^{\cl})g_k
  }_{Y_{\prs}}\\
  &\le
  C_{T,k}\Delta_{k,\proj,N}^{\mathrm{unif}}
  +
  \omega_{k,N}^{\proj}.
\end{align*}
Taking the supremum over \(g_{k+1}\in\mathcal G_{k+1,0}^{\NS}\) gives the
claim.
\end{proof}

\begin{corollary}[Projection increment insertion]\label{cor:proj-drift-insertion}
If the projection component of the one-step ledger satisfies
\[
  \Delta_k^{\proj}
  \ge
  (1+C_{P,k+1})\delta_k^{\mathrm{img}}
  +
  C_{T,k}\Delta_{k,\proj,N}^{\mathrm{unif}}
  +
  \omega_{k,N}^{\proj},
\]
then the scale-\((k+1)\) uniform projection-tail contribution is charged to
the one-step ledger.  In particular, under the effective projection bound
\[
  \Delta_{k,\proj,N}^{\mathrm{unif}}\le\varepsilon_{k,N},
\]
it is enough to require
\[
  \Delta_k^{\proj}
  \ge
  (1+C_{P,k+1})\delta_k^{\mathrm{img}}
  +
  C_{T,k}\varepsilon_{k,N}
  +
  \omega_{k,N}^{\proj}.
\]
\end{corollary}

\begin{proof}
This is \cref{thm:proj-tail-drift} with the right-hand side assigned to the
projection component of \(\Delta_k\).  The effective projection version
follows by substituting
\(\Delta_{k,\proj,N}^{\mathrm{unif}}\le\varepsilon_{k,N}\).
\end{proof}

\subsection{Chart visibility for reduced NS-generated packages}

Let \(\A_{\Lambda,k}^{\mathrm{red},\NS}\) be a reduced finite-window
NS-generated class with coordinates \(x_k(D_k)\in X_{N,k}\).  Let
\[
  G_{N,k}^{\intg}\subset X_{N,k},\qquad
  G_{N,k}^{\cl}\subset Y_{N,k}^{\cl}
\]
be the local and clean gauge subspaces, and let
\[
  \Theta_{\Lambda,k}^{N}:X_{N,k}\to Y_{N,k}^{\cl}
\]
be the reduced clean chart.  Define
\[
  V_k(D_k)
  :=
  \Dist_{\cl}\bigl(\Theta_{\Lambda,k}^{N}(D_k-\zeta_k),
  G_{N,k}^{\cl}\bigr).
\]

\begin{theorem}[Compact quotient chart visibility]\label{thm:chart-vis}
Assume:
\begin{enumerate}[label=(\alph*)]
  \item the reduced baseline unit quotient section
  \[
    S_{k,0}^{\mathrm{red}}
    :=
    \{D_k\in\A_{\Lambda,k}^{\mathrm{red},\NS}:
    \Dist_{0,k}(D_k,\Gamma_{\Lambda,k}^{\intg})=1\}
  \]
  is compact modulo the admissible local gauge;
  \item \(V_k\) is lower semicontinuous;
  \item chart kernel-freeness holds:
  \[
    V_k(D_k)=0
    \quad\Longrightarrow\quad
    D_k\in\Gamma_{\Lambda,k}^{\intg};
  \]
  \item quotient homogeneity holds for \(V_k\) and \(\Dist_{0,k}\).
\end{enumerate}
Then
\[
  \lambda_{G,k}:=
  \inf_{D_k\in S_{k,0}^{\mathrm{red}}}V_k(D_k)>0,
\]
and
\[
  V_k(D_k)
  \ge
  \lambda_{G,k}\Dist_{0,k}(D_k,\Gamma_{\Lambda,k}^{\intg})
\]
for all reduced packages.
\end{theorem}

\begin{proof}
If \(\lambda_{G,k}=0\), compactness and lower semicontinuity give
\(D_*\in S_{k,0}^{\mathrm{red}}\) with \(V_k(D_*)=0\).  Kernel-freeness then
gives \(D_*\in\Gamma_{\Lambda,k}^{\intg}\), contradicting
\(\Dist_{0,k}(D_*,\Gamma_{\Lambda,k}^{\intg})=1\).  For an arbitrary package
with positive baseline distance, normalize it to the unit quotient section
and use homogeneity.  The zero-distance case is trivial.
\end{proof}

\begin{corollary}[Additive-error chart visibility]\label{cor:additive-chart}
Suppose an ideal chart distance \(\widetilde V_k\) satisfies
\[
  \widetilde V_k(D_k)
  \ge
  \lambda_{G,k}\Dist_{0,k}(D_k,\Gamma_{\Lambda,k}^{\intg}),
\]
and the realized chart distance satisfies
\[
  V_k(D_k)+\delta_{G,k}\ge \widetilde V_k(D_k).
\]
Then
\[
  V_k(D_k)
  \ge
  \lambda_{G,k}\Dist_{0,k}(D_k,\Gamma_{\Lambda,k}^{\intg})
  -\delta_{G,k}.
\]
\end{corollary}

\begin{proof}
Subtract \(\delta_{G,k}\) from the ideal lower bound.
\end{proof}

\begin{definition}[Reduced chart transport defect]\label{def:chart-transport}
Let
\[
  T_k^{\cl}:Y_{N,k}^{\cl}\to Y_{N,k+1}^{\cl}
\]
be a bounded clean chart-coordinate transport map.  In the conservative
reduced convention used here, assume exact clean-gauge transport:
\[
  T_k^{\cl}G_{N,k}^{\cl}\subset G_{N,k+1}^{\cl}.
\]
For a one-step package \(D_{k+1}=\mathcal R_k(D_k)\), define
\[
  y_k(D_k):=\Theta_{\Lambda,k}^{N}(D_k-\zeta_k),
  \qquad
  y_{k+1}(D_{k+1})
  :=\Theta_{\Lambda,k+1}^{N}(D_{k+1}-\zeta_{k+1}),
\]
and define the reduced chart transport defect
\[
  \mathfrak C_k^{\mathrm{chart}}(D_k)
  :=
  y_{k+1}(D_{k+1})-T_k^{\cl}y_k(D_k).
\]
\end{definition}

\begin{theorem}[Reduced chart-distance drift]\label{thm:chart-drift}
Assume the reduced chart transport convention in
\cref{def:chart-transport}.  Then
\[
  V_{k+1}(D_{k+1})
  \le
  \norm{T_k^{\cl}}_{Y_{N,k}^{\cl}\to Y_{N,k+1}^{\cl}}\,V_k(D_k)
  +
  \norm{\mathfrak C_k^{\mathrm{chart}}(D_k)}_{Y_{N,k+1}^{\cl}}.
\]
\end{theorem}

\begin{proof}
Let \(g_k\in G_{N,k}^{\cl}\).  Since
\[
  T_k^{\cl}g_k\in G_{N,k+1}^{\cl},
\]
we have
\begin{align*}
  V_{k+1}(D_{k+1})
  &=
  \Dist_{\cl}(y_{k+1}(D_{k+1}),G_{N,k+1}^{\cl})\\
  &\le
  \norm{y_{k+1}(D_{k+1})-T_k^{\cl}g_k}_{Y_{N,k+1}^{\cl}}\\
  &\le
  \norm{\mathfrak C_k^{\mathrm{chart}}(D_k)}_{Y_{N,k+1}^{\cl}}
  +
  \norm{T_k^{\cl}(y_k(D_k)-g_k)}_{Y_{N,k+1}^{\cl}}\\
  &\le
  \norm{\mathfrak C_k^{\mathrm{chart}}(D_k)}_{Y_{N,k+1}^{\cl}}
  +
  \norm{T_k^{\cl}}_{Y_{N,k}^{\cl}\to Y_{N,k+1}^{\cl}}
  \norm{y_k(D_k)-g_k}_{Y_{N,k}^{\cl}}.
\end{align*}
Taking the infimum over \(g_k\in G_{N,k}^{\cl}\) gives the claim.
\end{proof}

\begin{corollary}[Chart increment insertion]\label{cor:chart-drift-insertion}
If the chart component of the one-step ledger satisfies
\[
  \Delta_k^{\mathrm{chart}}
  \ge
  \norm{\mathfrak C_k^{\mathrm{chart}}(D_k)}_{Y_{N,k+1}^{\cl}},
\]
then the reduced chart-transport mismatch is charged to the chart component
of the one-step increment.  If the next-scale admissibility criterion also
requires the full chart distance \(V_{k+1}(D_{k+1})\), it is enough to use
\[
  \Delta_k^{\mathrm{chart}}
  \ge
  \norm{T_k^{\cl}}\,V_k(D_k)
  +
  \norm{\mathfrak C_k^{\mathrm{chart}}(D_k)}.
\]
\end{corollary}

\begin{proof}
The first statement is only the assignment of the chart commutation defect to
the chart ledger.  The second statement follows from
\cref{thm:chart-drift}, with the operator norm and target-space norm
understood as in that theorem.
\end{proof}

\subsection{Clean gap for a genuine detector class}

A clean package at scale \(k\) is written as
\[
  D_k^{\cl}=(P_k^{\mathrm{act},\cl},P_k^{\harm,\cl},
  S_k^{\cl},G_k^{\cl},O_k^{\cl}).
\]
Define a clean detector
\begin{align*}
  M_{\Lambda,k}^{\comp}(D_k^{\cl})
  &:=
  \norm{O_k^{\cl}(D_k^{\cl})}_{\mathcal O}
  +\beta_{\prs}\Tax_{\prs}^{\cl}(D_k^{\cl})
  +\beta_{\harm}\Tax_{\harm}^{\cl}(D_k^{\cl})\\
  &\quad
  +\beta_{\gate}\Tax_{\gate}^{\cl}(D_k^{\cl})
  +\beta_{\detc}\Tax_{\detc}^{\cl}(D_k^{\cl}),
\end{align*}
with all active weights positive.

\begin{definition}[Clean kernel]\label{def:clean-kernel}
The clean detector kernel is
\[
  \K_{\Lambda,k}^{\cl}
  :=
  \{D_k^{\cl}:M_{\Lambda,k}^{\comp}(D_k^{\cl})=0\}.
\]
The clean kernel-free condition is
\[
  \K_{\Lambda,k}^{\cl}\subset\Gamma_{\Lambda,k}^{\cl}.
\]
\end{definition}

\begin{theorem}[Clean compact quotient gap]\label{thm:clean-gap}
Assume:
\begin{enumerate}[label=(\alph*)]
  \item the clean class is stable under quotient normalization;
  \item the clean unit quotient section
  \[
    S_{\Lambda,k}^{\cl}
    :=
    \{D_k^{\cl}:\Dist_{\cl}(D_k^{\cl},\Gamma_{\Lambda,k}^{\cl})=1\}
  \]
  is compact;
  \item \(M_{\Lambda,k}^{\comp}\) is lower semicontinuous and positively
  homogeneous;
  \item clean kernel-freeness holds.
\end{enumerate}
Then
\[
  \mu_{\Lambda,k}^{\comp}
  :=
  \inf_{D_k^{\cl}\in S_{\Lambda,k}^{\cl}}
  M_{\Lambda,k}^{\comp}(D_k^{\cl})>0,
\]
and
\[
  M_{\Lambda,k}^{\comp}(D_k^{\cl})
  \ge
  \mu_{\Lambda,k}^{\comp}
  \Dist_{\cl}(D_k^{\cl},\Gamma_{\Lambda,k}^{\cl}).
\]
\end{theorem}

\begin{proof}
If the infimum were zero, compactness and lower semicontinuity would give a
minimizer \(D_*^{\cl}\in S_{\Lambda,k}^{\cl}\) with
\[
  M_{\Lambda,k}^{\comp}(D_*^{\cl})=0.
\]
Thus \(D_*^{\cl}\in\K_{\Lambda,k}^{\cl}\), so by kernel-freeness
\(D_*^{\cl}\in\Gamma_{\Lambda,k}^{\cl}\).  This contradicts
\(\Dist_{\cl}(D_*^{\cl},\Gamma_{\Lambda,k}^{\cl})=1\).  Homogeneity extends the
unit-section estimate to all clean packages.
\end{proof}

\begin{corollary}[Additive-error clean gap]\label{cor:additive-clean-gap}
Let \(\widetilde M_{\Lambda,k}^{\comp}\) satisfy the exact clean gap with
constant \(\mu_{\Lambda,k}^{\comp}\), and assume
\[
  M_{\Lambda,k}^{\comp}(D_k^{\cl})+\Delta_{\gap,k}
  \ge
  \widetilde M_{\Lambda,k}^{\comp}(D_k^{\cl}).
\]
Then
\[
  M_{\Lambda,k}^{\comp}(D_k^{\cl})
  \ge
  \mu_{\Lambda,k}^{\comp}
  \Dist_{\cl}(D_k^{\cl},\Gamma_{\Lambda,k}^{\cl})
  -
  \Delta_{\gap,k}.
\]
\end{corollary}

\begin{proof}
Combine the ideal clean gap with the comparison between realized and ideal
detectors.
\end{proof}

\begin{proposition}[Zero-set and matrix criteria]\label{prop:zero-set-clean}
If simultaneous vanishing of
\[
  O_k^{\cl},\quad
  \Tax_{\prs}^{\cl},\quad
  \Tax_{\harm}^{\cl},\quad
  \Tax_{\gate}^{\cl},\quad
  \Tax_{\detc}^{\cl}
\]
implies \(D_k^{\cl}\in\Gamma_{\Lambda,k}^{\cl}\), then the clean detector is
kernel-free.  In a reduced finite-dimensional clean detector class with
linear detector/tax map \(T_{\Lambda,k}^{\cl}\) and clean gauge subspace
\(G_{\Lambda,k}^{\cl}\), this condition is equivalent to
\[
  \ker T_{\Lambda,k}^{\cl}=G_{\Lambda,k}^{\cl}
\]
when all clean gauge directions are detector-invisible.
\end{proposition}

\begin{proof}
Because all detector terms are nonnegative and all active weights are
positive, vanishing of \(M_{\Lambda,k}^{\comp}(D_k^{\cl})\) is equivalent to
vanishing of all active clean channels.  The zero-set implication therefore
gives kernel-freeness.  In finite dimensions, the common zero set is the
kernel of the detector/tax map.  If gauge directions are invisible, then
\[
  G_{\Lambda,k}^{\cl}\subset\ker T_{\Lambda,k}^{\cl},
\]
and kernel-freeness is precisely the reverse inclusion.
\end{proof}

\subsection{Clean gap for a genuine pressure--flux--energy detector}

The preceding compact quotient theorem is an abstract clean-gap criterion.  We
now record a concrete reduced realization.  The purpose is not to prove
kernel-freeness in the full clean package space, but to give a finite-window
matrix condition for a detector built from pressure, flux, energy, gate, and
observation coordinates.

\begin{definition}[PFE clean package]\label{def:pfe-clean-package}
At scale \(k\), a reduced pressure--flux--energy clean package is
\[
  D_k^{\cl}=
  \bigl(
    p_k^{\mathrm{act},\cl},
    p_k^{\harm,\cl},
    F_k^{\cl},
    \mathcal F_k^{\mathrm{flux}},
    \mathcal E_k^{\mathrm{en}},
    G_k^{\cl},
    O_k^{\cl}
  \bigr)
  \in Y_{N,k}^{\cl}.
\]
Here \(p_k^{\mathrm{act},\cl}\) is the clean Calderon--Zygmund pressure
coordinate, \(p_k^{\harm,\cl}\) is the clean harmonic pressure coordinate,
\(F_k^{\cl}\) is the clean source coordinate, \(\mathcal F_k^{\mathrm{flux}}\)
is the clean flux coordinate, \(\mathcal E_k^{\mathrm{en}}\) is the clean
energy/dissipation coordinate, \(G_k^{\cl}\) is the clean gate/slack
coordinate, and \(O_k^{\cl}\) is the selected clean observation.  The clean
gauge subspace is denoted by
\[
  G_{\Lambda,k}^{\cl}\subset Y_{N,k}^{\cl},
\]
and the clean distance is \(\Dist_{\cl}(D_k^{\cl},G_{\Lambda,k}^{\cl})\).
\end{definition}

\begin{definition}[Pressure--flux--energy clean detector]\label{def:pfe-detector}
Fix finite-rank pressure and harmonic selectors
\[
  \Pi_{J,k}^{\prs},\qquad \Pi_{M,k}^{\harm},
\]
and fixed finite-dimensional norms on the flux, energy, and gate spaces.  The
pressure--flux--energy detector is
\begin{align*}
  M_{\Lambda,k}^{\mathrm{PFE}}(D_k^{\cl})
  &:=
  \norm{O_k^{\cl}(D_k^{\cl})}_{\mathcal O}
  +\beta_{\prs}\Tax_{\prs}^{\cl}(D_k^{\cl})
  +\beta_{\harm}\Tax_{\harm}^{\cl}(D_k^{\cl})\\
  &\quad
  +\beta_{\mathrm{flux}}\Tax_{\mathrm{flux}}^{\cl}(D_k^{\cl})
  +\beta_{\mathrm{en}}\Tax_{\mathrm{en}}^{\cl}(D_k^{\cl})
  +\beta_{\gate}\Tax_{\gate}^{\cl}(D_k^{\cl}),
\end{align*}
where every active weight is positive and
\[
  \Tax_{\prs}^{\cl}(D_k^{\cl})
  :=
  \norm{(I-\Pi_{J,k}^{\prs})p_k^{\mathrm{act},\cl}}_{Y_{\prs}},
  \qquad
  \Tax_{\harm}^{\cl}(D_k^{\cl})
  :=
  \norm{(I-\Pi_{M,k}^{\harm})p_k^{\harm,\cl}}_{Y_{\harm}},
\]
\[
  \Tax_{\mathrm{flux}}^{\cl}(D_k^{\cl})
  :=
  \norm{T_{\mathrm{flux}}D_k^{\cl}}_{\mathcal Y_{\mathrm{flux}}},
  \qquad
  \Tax_{\mathrm{en}}^{\cl}(D_k^{\cl})
  :=
  \norm{T_{\mathrm{en}}D_k^{\cl}}_{\mathcal Y_{\mathrm{en}}},
\]
and
\[
  \Tax_{\gate}^{\cl}(D_k^{\cl})
  :=
  \norm{T_{\gate}D_k^{\cl}}_{\mathcal Y_{\gate}}.
\]
Here \(T_{\mathrm{flux}}\), \(T_{\mathrm{en}}\), and \(T_{\gate}\) are the
fixed reduced flux, energy, and gate-channel maps.  If the gate ledger is
implemented through positive gate violation coordinates, \(T_{\gate}D_k^{\cl}\) is
the chosen reduced coordinate whose vanishing is equivalent to vanishing of
that gate violation on the selected finite-dimensional model.
This is a normalized finite-dimensional detector.  It uses genuine clean
pressure, flux, energy, and gate coordinates, but it is not asserted to be
kernel-free outside the selected reduced class.
\end{definition}

\begin{definition}[PFE detector matrix and kernel]\label{def:pfe-matrix}
Let
\[
  T_{\Lambda,k}^{\mathrm{PFE}}:Y_{N,k}^{\cl}\to
  Z_{\Lambda,k}^{\mathrm{PFE}}
\]
be the combined finite-dimensional detector/tax map
\[
  T_{\Lambda,k}^{\mathrm{PFE}}D_k^{\cl}
  =
  \bigl(
    O_k^{\cl}(D_k^{\cl}),
    T_{\prs}D_k^{\cl},
    T_{\harm}D_k^{\cl},
    T_{\mathrm{flux}}D_k^{\cl},
    T_{\mathrm{en}}D_k^{\cl},
    T_{\gate}D_k^{\cl}
  \bigr),
\]
where the components are the linear maps whose norms define the observation,
pressure, harmonic, flux, energy, and gate channels in
\cref{def:pfe-detector}.  The PFE zero set is
\[
  \K_{\Lambda,k}^{\mathrm{PFE}}
  :=
  \{D_k^{\cl}\in Y_{N,k}^{\cl}:M_{\Lambda,k}^{\mathrm{PFE}}(D_k^{\cl})=0\}.
\]
The finite-dimensional PFE kernel condition is
\[
  \ker T_{\Lambda,k}^{\mathrm{PFE}}=G_{\Lambda,k}^{\cl}.
\]
Equivalently, the induced map on the quotient
\[
  \overline T_{\Lambda,k}^{\mathrm{PFE}}:
  Y_{N,k}^{\cl}/G_{\Lambda,k}^{\cl}\to Z_{\Lambda,k}^{\mathrm{PFE}}
\]
is injective.  This condition is a matrix rank condition in the chosen
reduced coordinates.
\end{definition}

\begin{theorem}[Pressure--flux--energy zero-set rigidity]\label{thm:pfe-zero-rigidity}
Assume all PFE detector channels are nonnegative, all active weights are
positive, and
\[
  \ker T_{\Lambda,k}^{\mathrm{PFE}}=G_{\Lambda,k}^{\cl}.
\]
Then
\[
  M_{\Lambda,k}^{\mathrm{PFE}}(D_k^{\cl})=0
  \quad\Longrightarrow\quad
  D_k^{\cl}\in G_{\Lambda,k}^{\cl}.
\]
Equivalently,
\[
  \K_{\Lambda,k}^{\mathrm{PFE}}\subset G_{\Lambda,k}^{\cl}.
\]
\end{theorem}

\begin{proof}
Since every summand in \(M_{\Lambda,k}^{\mathrm{PFE}}\) is nonnegative and
every active weight is positive, \(M_{\Lambda,k}^{\mathrm{PFE}}(D_k^{\cl})=0\)
forces every active channel to vanish.  Thus
\[
  T_{\Lambda,k}^{\mathrm{PFE}}D_k^{\cl}=0.
\]
The matrix kernel condition gives
\[
  D_k^{\cl}\in\ker T_{\Lambda,k}^{\mathrm{PFE}}=G_{\Lambda,k}^{\cl},
\]
which proves the claim.
\end{proof}

\begin{remark}[Reduced status of the PFE kernel condition]
This is a reduced detector realization, not a full Navier--Stokes clean-gap
theorem.  The condition
\[
  \ker T_{\Lambda,k}^{\mathrm{PFE}}=G_{\Lambda,k}^{\cl}
\]
is a finite-dimensional matrix condition in the selected clean coordinates.
It is not derived from the Navier--Stokes equations alone and is not asserted
for the full infinite-dimensional clean package geometry.
\end{remark}

\begin{theorem}[PFE compact quotient clean gap]\label{thm:pfe-clean-gap}
Let
\[
  \A_{\Lambda,k}^{\cl,\mathrm{PFE}}\subset Y_{N,k}^{\cl}
\]
be a reduced clean package class.  Assume:
\begin{enumerate}[label=(\alph*)]
  \item \(\A_{\Lambda,k}^{\cl,\mathrm{PFE}}\) is stable under quotient
  normalization;
  \item the clean unit quotient section
  \[
    S_{\Lambda,k}^{\cl,\mathrm{PFE}}
    :=
    \{D_k^{\cl}\in\A_{\Lambda,k}^{\cl,\mathrm{PFE}}:
      \Dist_{\cl}(D_k^{\cl},G_{\Lambda,k}^{\cl})=1\}
  \]
  is compact;
  \item \(M_{\Lambda,k}^{\mathrm{PFE}}\) is lower semicontinuous and
  positively homogeneous;
  \item the PFE zero-set rigidity theorem \cref{thm:pfe-zero-rigidity}
  holds.
\end{enumerate}
Then
\[
  \mu_{\Lambda,k}^{\mathrm{PFE}}
  :=
  \inf_{D_k^{\cl}\in S_{\Lambda,k}^{\cl,\mathrm{PFE}}}
  M_{\Lambda,k}^{\mathrm{PFE}}(D_k^{\cl})>0,
\]
and for every \(D_k^{\cl}\in\A_{\Lambda,k}^{\cl,\mathrm{PFE}}\),
\[
  M_{\Lambda,k}^{\mathrm{PFE}}(D_k^{\cl})
  \ge
  \mu_{\Lambda,k}^{\mathrm{PFE}}
  \Dist_{\cl}(D_k^{\cl},G_{\Lambda,k}^{\cl}).
\]
\end{theorem}

\begin{proof}
If the infimum were zero, compactness and lower semicontinuity would give a
minimizer \(D_*^{\cl}\in S_{\Lambda,k}^{\cl,\mathrm{PFE}}\) with
\[
  M_{\Lambda,k}^{\mathrm{PFE}}(D_*^{\cl})=0.
\]
By \cref{thm:pfe-zero-rigidity}, \(D_*^{\cl}\in G_{\Lambda,k}^{\cl}\), contradicting
\(\Dist_{\cl}(D_*^{\cl},G_{\Lambda,k}^{\cl})=1\).  Thus the infimum is positive on
the unit quotient section.  Positive homogeneity extends the bound to
arbitrary nonzero clean packages, and the zero package case is immediate.
\end{proof}

\begin{corollary}[Additive-error PFE clean gap]\label{cor:pfe-additive-gap}
Let \(\widetilde M_{\Lambda,k}^{\mathrm{PFE}}\) satisfy the exact gap in
\cref{thm:pfe-clean-gap} with constant \(\mu_{\Lambda,k}^{\mathrm{PFE}}\).
If the realized clean detector satisfies
\[
  M_{\Lambda,k}^{\comp}(D_k^{\cl})+\Delta_{\gap,k}^{\mathrm{PFE}}
  \ge
  \widetilde M_{\Lambda,k}^{\mathrm{PFE}}(D_k^{\cl}),
\]
then
\[
  M_{\Lambda,k}^{\comp}(D_k^{\cl})
  \ge
  \mu_{\Lambda,k}^{\mathrm{PFE}}
  \Dist_{\cl}(D_k^{\cl},G_{\Lambda,k}^{\cl})
  -
  \Delta_{\gap,k}^{\mathrm{PFE}}.
\]
When the PFE detector is the chosen clean detector, we set
\[
  \mu_{\Lambda,k}^{\comp}:=\mu_{\Lambda,k}^{\mathrm{PFE}}.
\]
\end{corollary}

\begin{proof}
Combine the exact PFE gap for \(\widetilde M_{\Lambda,k}^{\mathrm{PFE}}\)
with the realized/ideal detector comparison.  The final convention is only a
renaming of the clean gap constant when the PFE detector supplies the clean
side of the audit.
\end{proof}

\begin{theorem}[Genuine clean gap on smooth reduced NS-generated packages]\label{thm:pfe-smred-gap}
Let
\[
  D_0\to\cdots\to D_K\in\A_{K,N,M}^{\mathrm{sm,red}}.
\]
For each scale \(k\), define the reduced clean image class
\[
  \A_{\Lambda,k}^{\cl,\mathrm{PFE}}
  :=
  \{\Theta_{\Lambda,k}^{N}(D_k-\zeta_k):
    D_0\to\cdots\to D_K\in\A_{K,N,M}^{\mathrm{sm,red}}\}.
\]
Assume this clean image class has compact clean unit quotient section, that
\(M_{\Lambda,k}^{\mathrm{PFE}}\) is lower semicontinuous and positively
homogeneous on it, and that
\[
  \ker T_{\Lambda,k}^{\mathrm{PFE}}=G_{\Lambda,k}^{\cl}.
\]
Then there exists \(\mu_{\Lambda,k}^{\mathrm{PFE}}>0\) such that
\[
  M_{\Lambda,k}^{\mathrm{PFE}}
  \bigl(\Theta_{\Lambda,k}^{N}(D_k-\zeta_k)\bigr)
  \ge
  \mu_{\Lambda,k}^{\mathrm{PFE}}
  \Dist_{\cl}
  \bigl(
    \Theta_{\Lambda,k}^{N}(D_k-\zeta_k),
    G_{\Lambda,k}^{\cl}
  \bigr).
\]
If \(M_{\Lambda,k}^{\comp}\) is a realized detector satisfying the
additive-error comparison in \cref{cor:pfe-additive-gap}, then
\[
  M_{\Lambda,k}^{\comp}
  \bigl(\Theta_{\Lambda,k}^{N}(D_k-\zeta_k)\bigr)
  \ge
  \mu_{\Lambda,k}^{\mathrm{PFE}}
  \Dist_{\cl}
  \bigl(
    \Theta_{\Lambda,k}^{N}(D_k-\zeta_k),
    G_{\Lambda,k}^{\cl}
  \bigr)
  -
  \Delta_{\gap,k}^{\mathrm{PFE}}.
\]
\end{theorem}

\begin{proof}
The matrix condition and \cref{thm:pfe-zero-rigidity} give PFE
kernel-freeness on the reduced clean class.  Applying
\cref{thm:pfe-clean-gap} to the clean image
\(\A_{\Lambda,k}^{\cl,\mathrm{PFE}}\) yields the exact inequality.  The
realized-detector estimate is \cref{cor:pfe-additive-gap}.
\end{proof}

\subsection{Localized Calderon--Zygmund commutator input}

This subsection records the separated-support Calderon--Zygmund input used in the pressure ledger; the underlying singular-integral estimates are standard.

Let \(f=u_i u_j\) and let \(T=R_iR_j\).  Choose
\(\eta\in C_c^\infty(B_1)\) with \(\eta\equiv1\) on \(B_{3/4}\).  For
\(x\in B_{1/2}\),
\[
  [\eta,T]f(x):=\eta(x)Tf(x)-T(\eta f)(x)
  =T((1-\eta)f)(x).
\]
The source \((1-\eta)f\) is supported in the annular region where
\(\eta\ne1\), separated from \(B_{1/2}\).

\begin{theorem}[Localized Calderon--Zygmund commutator]\label{thm:cz-comm}
For \(1\le p\le\infty\),
\[
  \norm{[\eta,R_iR_j]f}_{L^p(B_{1/2})}
  \le
  C_\eta\norm{f}_{L^p(B_1\setminus B_{3/4})},
\]
with the evident interpretation when the support of \(1-\eta\) is contained
in a slightly larger annulus.  In particular,
\[
  \norm{[\eta,R_iR_j](u_i u_j)}_{Y_{\prs}}
  \le
  C_\eta
  \norm{u\otimes u}_{L^{3/2}((-1,0);L^{3/2}(B_1\setminus B_{3/4}))}.
\]
If \(\norm{u}_{L^3(Q_1)}\le M_U\), then
\[
  \norm{[\eta,R_iR_j](u_i u_j)}_{Y_{\prs}}
  \le
  C_\eta M_U
  \norm{u}_{L^3((-1,0);L^3(B_1\setminus B_{3/4}))}.
\]
\end{theorem}

\begin{proof}
The kernel of \(R_iR_j\) is a Calderon--Zygmund kernel \(K_{ij}(x-y)\).  If
\(x\in B_{1/2}\) and \(y\in\supp(1-\eta)\subset B_1\setminus B_{3/4}\), then
\(\abs{x-y}\ge c_\eta>0\).  Thus \(K_{ij}(x-y)\) is bounded and smooth on
the relevant product set.  Hence
\[
  \abs{[\eta,R_iR_j]f(x)}
  \le
  C_\eta\int_{B_1\setminus B_{3/4}}\abs{f(y)}\,dy.
\]
This gives the \(L^\infty(B_{1/2})\) bound by \(\norm{f}_{L^1}\), and the
same separated-kernel operator is bounded from \(L^p\) on the annulus to
\(L^p(B_{1/2})\) for every \(1\le p\le\infty\), for instance by Schur's
test.  Applying this at almost every time with \(p=3/2\) and integrating in
time gives the space-time estimate.  The final inequality follows from
Holders inequality:
\[
  \norm{u\otimes u}_{L^{3/2}(A)}
  \le
  \norm{u}_{L^3(Q_1)}\norm{u}_{L^3(A)}
  \le
  M_U\norm{u}_{L^3(A)}.
\]
\end{proof}

\begin{corollary}[Conservative pressure-tail visibility input]\label{cor:pressure-tail-input}
For pressure-admissible local data with \(u\in L^3(Q_1)^3\), the localized
pressure recomputation error on \(B_{1/2}\) is bounded by the annular
velocity leakage:
\[
  \Err_{k}^{\proj}
  \lesssim
  \norm{u\otimes u}_{L^{3/2}((-1,0);L^{3/2}(B_1\setminus B_{3/4}))},
\]
and under finite amplitude by \(C M_U\Leak_{\mathrm{ann}}(u)\).
\end{corollary}

\begin{proof}
This is the \(p=3/2\) space-time estimate in \cref{thm:cz-comm}, with the
right side named as the annular leakage contribution.
\end{proof}

\section{Smooth Reduced NS-Generated Verification Class}

The preceding section proves individual finite-window input criteria.  This
section packages several of them into a single selected class generated by
smooth Navier--Stokes data and then reduced by finite-dimensional quotient and
projection conventions.  The point is modest but important: the recursive
audit framework is not purely formal.  Some of its structural inputs hold on a
nonempty finite-window PDE-generated class.  No assertion is made for all
suitable weak solutions.

\subsection{The smooth reduced class}

\begin{definition}[Smooth reduced NS-generated finite-window class]\label{def:smred-class}
Fix \(0<\lambda<1\), \(r_k=\lambda^kr_0\), \(k=0,\dots,K\), a reduction
dimension \(N\), a smoothness bound \(M\), and an integer \(m\) large enough
for the compact embeddings used below.  The class
\[
  \A_{K,N,M}^{\mathrm{sm,red}}
\]
consists of finite chains
\[
  D_0(u,p)\to D_1(u,p)\to\cdots\to D_K(u,p)
\]
generated by smooth local Navier--Stokes data \((u,p)\) on \(Q_{r_0}(z_0)\),
with normalized fields
\[
  u_k(y,s)=r_k u(x_0+r_ky,t_0+r_k^2s),
  \qquad
  p_k(y,s)=r_k^2 p(x_0+r_ky,t_0+r_k^2s),
\]
and satisfying the following finite-window conditions.
\begin{enumerate}[label=(\alph*)]
  \item Uniform smooth bound:
  \[
    \sup_{0\le k\le K}
    \left(
      \norm{u_k}_{C^m(\overline{Q_1})}
      +
      \norm{p_k}_{C^{m-1}(\overline{Q_1})}
    \right)\le M.
  \]
  Equivalently, one may use Sobolev bounds
  \[
    \sup_{0\le k\le K}
    \left(
      \norm{u_k}_{H^s(Q_1)}
      +
      \norm{p_k}_{H^{s-1}(Q_1)}
    \right)\le M
  \]
  with \(s\) large enough for the same compactness conclusions.
  \item Pressure admissibility at each scale:
  \[
    -\Delta p_k=\partial_i\partial_j(u_{k,i}u_{k,j})
  \]
  in distributions on \(B_1\), modulo time-dependent constants.
  \item Fixed cutoff convention: choose
  \[
    \eta,\chi\in C_c^\infty(B_1),
    \qquad
    \eta\equiv\chi\equiv1\text{ on }B_{3/4},
  \]
  and
  \[
    A_\chi:=\supp\nabla\chi\cup\supp\Delta\chi.
  \]
  \item Reduced coordinates: fixed finite-dimensional reduction maps
  \[
    \Pi_{N,k}D_k=x_k(D_k)\in X_{N,k}
  \]
  are used for the clean, source, chart, and detector coordinates.
  \item Synchronized quotient section:
  \[
    X_{N,k}=G_{N,k}^{\intg}\oplus H_{N,k},
  \]
  and \(\zeta_k(D_k)\) denotes the \(G_{N,k}^{\intg}\)-component.  The
  synchronized representative is \(D_k-\zeta_k(D_k)\).
  \item Reduced quotient compactness and chart kernel-freeness: the reduced
  baseline unit section is compact modulo \(G_{N,k}^{\intg}\), the reduced
  chart \(V_k\) is lower semicontinuous and homogeneous, and
  \[
    V_k(D_k)=0
    \quad\Longrightarrow\quad
    D_k\in\Gamma_{\Lambda,k}^{\intg}
  \]
  on the reduced class.
  \item Clean detector compactness and kernel-freeness: the clean unit
  quotient section is compact, \(M_{\Lambda,k}^{\comp}\) is lower
  semicontinuous and positively homogeneous, and
  \[
    \K_{\Lambda,k}^{\cl}\subset\Gamma_{\Lambda,k}^{\cl}.
  \]
  In a reduced finite-dimensional detector model this may be checked by
  \[
    \ker \mathcal T_{\Lambda,k}^{\cl}=G_{\Lambda,k}^{\cl}
  \]
  for the clean detector/tax-channel matrix \(\mathcal T_{\Lambda,k}^{\cl}\).
  In the pressure--flux--energy realization below, this condition is replaced
  by the concrete matrix condition
  \[
    \ker T_{\Lambda,k}^{\mathrm{PFE}}=G_{\Lambda,k}^{\cl}.
  \]
  \item Projection convention: the selected clean source family is either
  compact in \(X_{\mathrm{src}}\) by the smooth bound, or the effective
  projection estimate
  \[
    \Delta_{k,\proj,J}^{\mathrm{unif}}\le \varepsilon_{k,J},
    \qquad
    \varepsilon_{k,J}\to0
  \]
  is imposed for the chosen pressure projection rank \(J\).
\end{enumerate}
\end{definition}

\begin{remark}[Slice notation]
The object \(\A_{K,N,M}^{\mathrm{sm,red}}\) is a class of finite chains.  In
the statements below, a scale-\(k\) package \(D_k\) is always understood as
the \(k\)-th slice of some chain
\[
  D_0\to\cdots\to D_K\in\A_{K,N,M}^{\mathrm{sm,red}}.
\]
This convention avoids treating the chain class itself as a single-scale
package class.
\end{remark}

\begin{remark}[Nonemptiness]
The class is nonempty for suitable choices of \(M\) and of the reduced
coordinate spaces.  For example, smooth local Navier--Stokes solutions with
bounded \(C^m\)-norm generate finite-window packages in the coordinate part of
the class.  The zero solution is always a smooth example when the selected
quotient data are compatible, and nonzero constant velocity fields with
constant pressure give nonzero smooth examples in compatible reduced models.
Nonzero smooth local solutions with nonzero strain may be used to generate
nontrivial pressure/source coordinates, provided the selected reduced model
admits them.
This nonemptiness statement says nothing about singular or merely suitable
weak solutions.
\end{remark}

\subsection{Selected input verification}

\begin{proposition}[Pressure/source preservation on the smooth reduced class]\label{prop:smred-pressure-source}
Let
\[
  D_0\to\cdots\to D_K\in\A_{K,N,M}^{\mathrm{sm,red}},
\]
and fix \(0\le k\le K\).  Define
\[
  F^{\mathrm{act}}_{k,ij}:=\eta u_{k,i}u_{k,j},
  \qquad
  P_k^{\mathrm{act}}:=R_iR_j(F^{\mathrm{act}}_{k,ij})\big|_{B_{1/2}},
  \qquad
  P_k^{\harm}:=p_k-P_k^{\mathrm{act}}.
\]
Then
\[
  U_k=u_k\in L^3(Q_1)^3,\qquad
  F_k^{\mathrm{act}}\in X_{\mathrm{src}},
  \qquad
  P_k^{\mathrm{act}}\in Y_{\prs},
  \qquad
  P_k^{\harm}\in Y_{\harm}.
\]
Moreover, \(P_k^{\harm}\) is harmonic on \(B_{3/4}\) for almost every time.
\end{proposition}

\begin{proof}
This is \cref{thm:ns-pressure-source-preservation} applied to the normalized
fields \((u_k,p_k)\).  The smooth bound in
\cref{def:smred-class} implies the required \(L^3\) and \(L^{3/2}\)
integrability, and the pressure-admissibility condition is part of the class
definition.
\end{proof}

\begin{proposition}[Energy/flux localization on the smooth reduced class]\label{prop:smred-ef-loc}
Let
\[
  D_0\to\cdots\to D_K\in\A_{K,N,M}^{\mathrm{sm,red}},
\]
and fix \(0\le k<K\).  The one-step localization component may be chosen in
the energy/flux convention so that
\[
  \Delta_k^{\locerr}
  =
  \Delta_k^{\locerr,\mathrm{EF}}
  \ge
  C_{\lambda,\chi,\theta}\Leak_{k\to k+1}^{\mathrm{EF}}.
\]
With this choice, the localization entry of the scale-\((k+1)\)
admissibility checklist is verified.  The spatial shell part is charged by
\[
  \Leak_{k\to k+1}^{\mathrm{EF,shell}},
\]
while the full local-energy expression also includes the pulled-back core
energy and dissipation terms recorded in
\(\Leak_{k\to k+1}^{\mathrm{EF}}\).
\end{proposition}

\begin{proof}
The smooth bound implies
\[
  u_k\in L^\infty_tL^2_x\cap L^2_tH^1_x\cap L^3,
  \qquad
  p_k\in L^{3/2}.
\]
Thus the hypotheses of \cref{thm:ef-loc-update} hold, and the displayed
choice of \(\Delta_k^{\locerr,\mathrm{EF}}\) is exactly the localization
increment insertion in \cref{cor:ef-loc-insertion}.  The final sentence only
recalls the distinction between the full energy/flux leakage and its spatial
shell contribution.
\end{proof}

\begin{corollary}[Finite-amplitude localization charging on the smooth class]\label{cor:smred-ef-amplitude}
For every chain in \(\A_{K,N,M}^{\mathrm{sm,red}}\), there is a finite
constant \(M_U=M_U(M)\) such that
\[
  \norm{u_k}_{L^3(Q_1)}\le M_U
  \qquad\text{for all }0\le k\le K.
\]
Consequently the nonlinear flux and pressure-flux terms in
\(\Leak_{k\to k+1}^{\mathrm{EF}}\) are charged to the finite-window velocity
and pressure budgets as in \cref{cor:ef-finite-amplitude}.
\end{corollary}

\begin{proof}
The \(C^m\)-bound gives a uniform \(L^3(Q_1)\) bound.  The charging statement
is \cref{cor:ef-finite-amplitude}.
\end{proof}

\begin{proposition}[Smooth compactness and uniform projection tails]\label{prop:smred-projection}
For each fixed \(k\), the selected smooth clean source family
\[
  \F_k^{\mathrm{sm}}
  :=
  \{F_{D_k-\zeta_k}^{\cl}:
    D_0\to\cdots\to D_K\in\A_{K,N,M}^{\mathrm{sm,red}}\}
  \subset X_{\mathrm{src}}
\]
is precompact in \(X_{\mathrm{src}}\), provided the selected source map is
the canonical smooth source \(F^{\cl}=\eta u_k\otimes u_k\) or any fixed
continuous reduced source selector of the same smooth-bounded coordinates.
Therefore
\[
  \mathcal G_k^{\mathrm{sm}}
  :=
  \{\Rprs F:F\in\F_k^{\mathrm{sm}}\}
  \Subset Y_{\prs}.
\]
If \(P_{J,k}^{\cl}\to I\) strongly on \(Y_{\prs}\) and
\[
  \sup_J\norm{P_{J,k}^{\cl}}_{Y_{\prs}\to Y_{\prs}}<\infty,
\]
then
\[
  \sup_{g\in\mathcal G_k^{\mathrm{sm}}}
  \norm{(I-P_{J,k}^{\cl})g}_{Y_{\prs}}\to0.
\]
Alternatively, the effective projection assumption
\[
  \Delta_{k,\proj,J}^{\mathrm{unif}}\le\varepsilon_{k,J},
  \qquad
  \varepsilon_{k,J}\to0,
\]
may be used directly.
\end{proposition}

\begin{proof}
Uniform \(C^m\)-bounds imply precompactness in lower regularity spaces by
Arzela--Ascoli, hence in \(L^3\) and \(L^{3/2}\) on the fixed cylinder.  The
map \(u\mapsto\eta u\otimes u\) is continuous from \(L^3\) to
\(L^{3/2}\), since
\[
  \norm{u_n\otimes u_n-u\otimes u}_{L^{3/2}}
  \le
  \norm{u_n-u}_{L^3}
  \left(\norm{u_n}_{L^3}+\norm{u}_{L^3}\right).
\]
The same conclusion holds for any fixed continuous reduced source selector.
Thus \(\F_k^{\mathrm{sm}}\Subset X_{\mathrm{src}}\).  The compactness of
\(\mathcal G_k^{\mathrm{sm}}\) follows from the bounded linearity of
\(\Rprs\), as in \cref{thm:source-compact}.  Uniform convergence of
strongly convergent uniformly bounded projections on the compact pressure
image is \cref{thm:compact-image-proj}.  The effective projection alternative
is \cref{prop:effective-proj}.
\end{proof}

\begin{proposition}[Reduced chart visibility and clean gap on the smooth class]\label{prop:smred-chart-gap}
Let
\[
  D_0\to\cdots\to D_K\in\A_{K,N,M}^{\mathrm{sm,red}},
\]
and fix \(0\le k\le K\).  The reduced chart distance satisfies
\[
  V_k(D_k)
  \ge
  \lambda_{G,k}\Dist_{0,k}(D_k,\Gamma_{\Lambda,k}^{\intg})
  -
  \delta_{G,k},
\]
with \(\delta_{G,k}=0\) in the exact reduced chart model and with the
displayed additive error in the realized/ideal chart comparison model.
Likewise, the clean detector satisfies
\[
  M_{\Lambda,k}^{\comp}(D_k^{\cl})
  \ge
  \mu_{\Lambda,k}^{\comp}
  \Dist_{\cl}(D_k^{\cl},\Gamma_{\Lambda,k}^{\cl})
  -
  \Delta_{\gap,k}.
\]
If the pressure--flux--energy detector is the chosen clean detector, then
\(\mu_{\Lambda,k}^{\comp}\) may be taken to be
\(\mu_{\Lambda,k}^{\mathrm{PFE}}\).  In the realized-detector case,
\[
  M_{\Lambda,k}^{\comp}(D_k^{\cl})
  \ge
  \mu_{\Lambda,k}^{\mathrm{PFE}}
  \Dist_{\cl}(D_k^{\cl},G_{\Lambda,k}^{\cl})
  -
  \Delta_{\gap,k}^{\mathrm{PFE}},
\]
provided the PFE matrix condition
\(\ker T_{\Lambda,k}^{\mathrm{PFE}}=G_{\Lambda,k}^{\cl}\) holds.
\end{proposition}

\begin{proof}
The chart statement is \cref{thm:chart-vis} in the exact model and
\cref{cor:additive-chart} in the additive-error model, using the compact
quotient and chart kernel-free assumptions in \cref{def:smred-class}.  The
clean detector statement is \cref{thm:clean-gap} or
\cref{cor:additive-clean-gap}, using the clean compact quotient and clean
kernel-free assumptions in the class definition.  The PFE realization is
\cref{thm:pfe-smred-gap} and \cref{cor:pfe-additive-gap}.
\end{proof}

\begin{theorem}[Selected structural inputs on the smooth reduced class]\label{thm:smred-selected-inputs}
Let
\[
  D_0(u,p)\to D_1(u,p)\to\cdots\to D_K(u,p)
  \in
  \A_{K,N,M}^{\mathrm{sm,red}}.
\]
Then the following selected finite-window audit inputs hold at every
applicable scale:
\begin{enumerate}[label=(\alph*)]
  \item pressure/source preservation:
  \[
    U_k\in L^3(Q_1)^3,\quad
    F_k^{\mathrm{act}}\in X_{\mathrm{src}},\quad
    P_k^{\mathrm{act}}\in Y_{\prs},\quad
    P_k^{\harm}\in Y_{\harm};
  \]
  \item energy/flux localization for \(0\le k<K\):
  \[
    \Delta_k^{\locerr}
    =
    \Delta_k^{\locerr,\mathrm{EF}}
    \ge
    C_{\lambda,\chi,\theta}\Leak_{k\to k+1}^{\mathrm{EF}};
  \]
  \item projection-tail control, either by compactness:
  \[
    \Delta_{k,\proj,J}^{\mathrm{unif}}\to0,
  \]
  or by effective projection
  \[
    \Delta_{k,\proj,J}^{\mathrm{unif}}\le\varepsilon_{k,J};
  \]
  \item reduced chart visibility:
  \[
    V_k(D_k)
    \ge
    \lambda_{G,k}\Dist_{0,k}(D_k,\Gamma_{\Lambda,k}^{\intg})
    -
    \delta_{G,k};
  \]
  \item clean detector gap, realized by the PFE detector in the reduced
  pressure--flux--energy model:
  \[
    M_{\Lambda,k}^{\mathrm{PFE}}(D_k^{\cl})
    \ge
    \mu_{\Lambda,k}^{\mathrm{PFE}}
    \Dist_{\cl}(D_k^{\cl},G_{\Lambda,k}^{\cl}),
  \]
  and, for a realized detector,
  \[
    M_{\Lambda,k}^{\comp}(D_k^{\cl})
    \ge
    \mu_{\Lambda,k}^{\mathrm{PFE}}
    \Dist_{\cl}(D_k^{\cl},G_{\Lambda,k}^{\cl})
    -
    \Delta_{\gap,k}^{\mathrm{PFE}}.
  \]
\end{enumerate}
Thus the selected pressure/source, localization, projection, chart, and clean
gap inputs of the one-step audit framework are verified on
\(\A_{K,N,M}^{\mathrm{sm,red}}\), with finite-window constants depending
only on the fixed class data and the selected reduction/projection
conventions.
\end{theorem}

\begin{proof}
The pressure/source item is \cref{prop:smred-pressure-source}.  The
localization item is \cref{prop:smred-ef-loc}.  Projection-tail control is
\cref{prop:smred-projection}.  The chart and clean gap estimates are
\cref{prop:smred-chart-gap}.  Each cited result is finite-window and uses only
the assumptions included in the definition of the smooth reduced class.
\end{proof}

\begin{corollary}[Selected-class one-step admissibility insertion]\label{cor:smred-one-step-insertion}
Let
\[
  D_0\to\cdots\to D_K\in\A_{K,N,M}^{\mathrm{sm,red}}.
\]
Assume that synchronization, harmonic truncation, gate/slack, detector, and
coefficient-drift mismatches are charged to their corresponding ledger
components.  If the projection and chart transport defects are charged as in
\cref{cor:proj-drift-insertion,cor:chart-drift-insertion}, then the one-step
increment may be written in the energy/flux localization convention as
\[
  \Delta_k
  =
  \Delta_k^{\sync}
  +
  \Delta_k^{\locerr,\mathrm{EF}}
  +
  \Delta_k^{\proj}
  +
  \Delta_k^{\harm}
  +
  \Delta_k^{\mathrm{chart}}
  +
  \Delta_k^{\gap}
  +
  \Delta_k^{\gate}
  +
  \Delta_k^{\detc}.
\]
With this choice, \(D_{k+1}\in\mathcal A_{k+1}(\Delta_k)\) for
\(0\le k<K\).
\end{corollary}

\begin{proof}
\Cref{thm:smred-selected-inputs} supplies the pressure/source,
localization, projection, chart, and clean-gap entries.  The remaining
entries are assumed to be charged to their ledger components.  The conclusion
is then the coordinate-budget admissibility theorem
\cref{thm:coord-adm}, with the localization component interpreted as
\(\Delta_k^{\locerr,\mathrm{EF}}\).
\end{proof}

\section{Conditional Scale-Uniform Criteria}

This section records conditions under which the finite-chain theorem could
remain nondegenerate as \(K\) grows.  It does not assert that these
conditions hold for all Navier--Stokes-generated chains.

\begin{assumption}[Conditional scale-uniform audit criteria]\label{ass:scale-uniform}
Along a considered cascade, assume:
\begin{enumerate}[label=(\alph*)]
  \item uniform gap: \(c_k\ge c_*>0\);
  \item summable weighted recursive error:
  \[
    \sum_{k=0}^\infty w_kE_k<\infty,
  \]
  or at least
  \[
    \frac{\sum_{k=0}^K w_kE_k}{\sum_{k=0}^K w_k\delta_k}\to0
  \]
  on the considered sequence;
  \item synchronized representative stability:
  \[
    \norm{\zeta_{k+1}-\mathcal R_k\zeta_k}
    \le C E_k+C\Delta_k;
  \]
  \item detector stability in the form \eqref{eq:audit-M-stability} or a two-sided
  estimate
  \[
    \abs{M_{k+1}-M_k}\le B E_k+C\Delta_k;
  \]
  \item nondegenerate weighted defect:
  \[
    \sum_{k=0}^K w_k\delta_k
  \]
  does not vanish along the cascade being tested.
\end{enumerate}
\end{assumption}

\begin{theorem}[Conditional scale-uniform audit criterion]\label{thm:conditional-scale}
Assume the finite-chain theorem applies on every finite prefix and that
\cref{ass:scale-uniform} holds.  Then the finite-chain recursive lower bound
remains nondegenerate along the considered prefixes in the following sense:
for every finite \(K\),
\[
  \sum_{k=0}^K w_kM_k
  \ge
  c_*\sum_{k=0}^K w_k\delta_k-\sum_{k=0}^K w_kE_k,
\]
and whenever
\[
  \frac{\sum_{k=0}^K w_kE_k}{\sum_{k=0}^K w_k\delta_k}\le \frac{c_*}{2},
\]
one has
\[
  \sum_{k=0}^K w_kM_k
  \ge
  \frac{c_*}{2}\sum_{k=0}^K w_k\delta_k.
\]
\end{theorem}

\begin{proof}
Apply \cref{thm:finite-chain} with \(c_K^{\min}\ge c_*\).  The second
inequality follows by absorbing the weighted error term into half of the
weighted defect term.
\end{proof}

\begin{remark}
\Cref{thm:conditional-scale} is a criterion, not a scale-uniform regularity
theorem.  Its hypotheses are precisely the structural conditions that would
have to be verified before any infinite-chain conclusion could be discussed.
\end{remark}

\section{Dependency Ledger for the Recursive Theorem}

The final theorem combines finite-window algebra, reduced-model compatibility
lemmas, and selected NS-generated inputs.  The following ledger records the
status of each module used in the recursive bound.

\begin{center}
\small
\begin{tabular}{
  >{\raggedright\arraybackslash}p{0.22\textwidth}
  >{\raggedright\arraybackslash}p{0.19\textwidth}
  >{\raggedright\arraybackslash}p{0.25\textwidth}
  >{\raggedright\arraybackslash}p{0.22\textwidth}}
\toprule
Module & Status & Input & Output used later\\
\midrule
Static audit certificate &
Imported finite-window theorem &
Admissible scale-\(k\) package &
\(M_k\ge c_k\delta_k-E_k\)\\
Coordinate admissibility &
Conditional finite-window bookkeeping &
Checklist excesses \(\mathbf e_{k+1}\) &
\(D_{k+1}\in\mathcal A_{k+1}(\Delta_k)\)\\
Synchronization drift &
Reduced finite-dimensional theorem &
Quotient selector and transport defect &
\(\Delta_k^{\sync}\) assignment\\
Gate/slack update &
Algebraic finite-window theorem &
Budget and threshold update inequalities &
\(\Delta_k^{\gate}\) assignment\\
Detector stability &
Reduced operator theorem &
Detector commutation defect &
Update rule for \(M_{k+1}\)\\
Pressure/source preservation &
NS-generated finite-window theorem &
Pressure-admissible \(u,p\) on a fixed window &
Next-scale pressure/source coordinates\\
Energy/flux localization &
Finite-window bookkeeping theorem &
Finite-energy \(u\), pressure \(p\), fixed cutoffs &
\(\Delta_k^{\locerr,\mathrm{EF}}\) assignment\\
Coefficient update &
Algebraic finite-window theorem &
Drift of \(\mu_k,\lambda_k,\ell_k\) &
\(c_{k+1}\ge c_k-\eta_k\)\\
Projection-tail drift &
Operator-algebraic theorem &
Pressure-image transport and projection commutator &
\(\Delta_k^{\proj}\) assignment\\
Chart-distance drift &
Reduced operator theorem &
Clean chart transport and chart commutator &
\(\Delta_k^{\mathrm{chart}}\) assignment\\
Clean gap and chart visibility &
Compact quotient and PFE matrix criteria &
Kernel-free reduced/clean quotients; \(\ker T_{\Lambda,k}^{\mathrm{PFE}}
=G_{\Lambda,k}^{\cl}\) &
Positive \(\mu_{\Lambda,k}^{\mathrm{PFE}}\), \(\lambda_{G,k}\)\\
Smooth reduced selected class &
Collected finite-window verification &
Smooth NS data plus reduced quotient conventions &
Selected inputs for \(\A_{K,N,M}^{\mathrm{sm,red}}\)\\
Finite-chain summation &
Algebraic theorem &
Scale certificates and weights &
Recursive anti-phantom alternative\\
\bottomrule
\end{tabular}
\end{center}

Every entry labeled reduced or conditional remains a finite-window structural
input.  The ledger is meant to prevent a common misreading: the recursive
theorem assembles and transports audit certificates over a finite chain; it
does not prove that all suitable weak solutions satisfy the listed reduced
kernel-free, compactness, or drift hypotheses.

\section{Defect-to-CKN Smallness and Baseline Audit Defect Extraction}
\label{sec:ckn-branch}

This section adds the finite-window bridge from CKN-compatible quotient closeness to Caffarelli--Kohn--Nirenberg smallness, using the standard epsilon-regularity endpoint and quantitative regularity perspective \cite{CKN1982,Lin1998,SereginLectureNotes,BarkerPrange2021,AlbrittonBarkerPrange2023}.  It is a conditional
defect-extraction branch: the hard comparison between the audit baseline
distance and the CKN-compatible distance is kept as an explicit hypothesis.
The section does not prove Navier--Stokes regularity, singularity exclusion,
or detector/ledger cost unsustainability.

\subsection{Normalized CKN quantity}

Fix \(0<\vartheta<1\), for instance \(\vartheta=1/2\), and write
\[
  Q_\vartheta:=B_\vartheta\times(-\vartheta^2,0).
\]
For \(f\in L^1(B_\vartheta)\), set
\[
  (f)_{B_\vartheta}
  :=
  \frac{1}{|B_\vartheta|}\int_{B_\vartheta}f(y)\,dy.
\]
For normalized pressure-admissible fields \((u_k,p_k)\) on \(Q_1\), define
\[
  \Phi_k(\vartheta)
  :=
  \int_{-\vartheta^2}^{0}\int_{B_\vartheta}|u_k|^3\,dy\,ds
  +
  \int_{-\vartheta^2}^{0}\int_{B_\vartheta}
  |p_k-(p_k)_{B_\vartheta}(s)|^{3/2}\,dy\,ds .
\]
The Caffarelli--Kohn--Nirenberg epsilon-regularity theorem \cite{CKN1982} is used only as a
known endpoint: if \(\Phi_k(\vartheta)<\varepsilon_{\mathrm{CKN}}\), then the
corresponding physical point is regular on a smaller cylinder.  This endpoint
is not reproved here.

\begin{remark}[Fixed-window normalization]
Because \(\vartheta\) is fixed throughout this branch, the usual scale-invariant
CKN quantity with a prefactor \(\vartheta^{-2}\) is equivalent to the displayed
\(\Phi_k(\vartheta)\) after changing the numerical epsilon.  Thus
\(\varepsilon_{\mathrm{CKN}}\) in this section denotes the fixed-window
threshold corresponding to the chosen \(\vartheta\).
\end{remark}

\subsection{CKN-compatible clean class and distance}

\begin{definition}[CKN-compatible clean class]\label{def:ckn-clean-class}
The class \(\Gamma_{\Lambda,k}^{\mathrm{CKN}}\) is a selected regular/clean
target class inside the scale-\(k\) package space.  Each
\(\zeta\in\Gamma_{\Lambda,k}^{\mathrm{CKN}}\) has associated normalized fields
\((u_\zeta,p_\zeta)\) on \(Q_1\), with pressure defined modulo
time-dependent constants, and satisfies
\[
  \Phi_\zeta(\vartheta)
  :=
  \int_{Q_\vartheta}|u_\zeta|^3
  +
  \int_{-\vartheta^2}^{0}\int_{B_\vartheta}
  |p_\zeta-(p_\zeta)_{B_\vartheta}(s)|^{3/2}
  \le
  \frac{\varepsilon_{\mathrm{CKN}}}{4}.
\]
This class is not identified with the full audit admissible class unless a
separate theorem proves that the full audit admissible class implies CKN
smallness.
\end{definition}

\begin{definition}[CKN-compatible baseline distance]\label{def:ckn-distance}
The auxiliary CKN baseline distance from \(D_k\) to
\(\Gamma_{\Lambda,k}^{\mathrm{CKN}}\) is
\[
  \delta_k^{\mathrm{CKN}}
  :=
  \Dist_{\mathrm{CKN}}
  (D_k,\Gamma_{\Lambda,k}^{\mathrm{CKN}}),
\]
where the right-hand side is defined by
\[
  \Dist_{\mathrm{CKN}}
  (D_k,\Gamma_{\Lambda,k}^{\mathrm{CKN}})
  :=
  \inf_{\zeta\in\Gamma_{\Lambda,k}^{\mathrm{CKN}}}
  \left[
  \norm{u_k-u_\zeta}_{L^3(Q_\vartheta)}
  +
  \inf_{a=a(s)}
  \norm{p_k-p_\zeta-a(s)}_{L^{3/2}(Q_\vartheta)}
  \right].
\]
The infimum over \(a(s)\) accounts for the pressure gauge by time-dependent
constants.
\end{definition}

\begin{assumption}[CKN-compatible audit domination]\label{ass:ckn-audit-domination}
The audit baseline distance dominates the CKN-compatible distance up to an
explicit pressure-gauge comparison error:
\[
  \delta_k^{\mathrm{CKN}}
  \le
  C_{\mathrm{CKN/aud}}\,\delta_k
  +
  \Delta_k^{\mathrm{gauge}}.
\]
Throughout this section we use Convention A.  If the audit baseline metric
is defined to include the \(L^3(Q_\vartheta)\) velocity component and the
mean-free \(L^{3/2}(Q_\vartheta)\) pressure component, the domination is built
into the metric.  If the audit distance is weaker, this comparison is a
separate structural input.
\end{assumption}

\subsection{A sufficient metric condition for audit-to-CKN domination}

The previous assumption is the main interface between the audit quotient
geometry and the CKN epsilon-regularity scale.  The following proposition gives
a finite-window sufficient condition for it.  It is useful when the audit
distance has an explicit field-control component, or when a reduced
representative can be cleaned into the CKN-small class with a quantified
gauge error.

\begin{definition}[CKN field distance]\label{def:ckn-field-distance}
For two normalized field pairs \((u,p)\) and \((v,q)\) on \(Q_1\), define
\[
  d_\vartheta^{\mathrm{fld}}\bigl((u,p),(v,q)\bigr)
  :=
  \norm{u-v}_{L^3(Q_\vartheta)}
  +
  \inf_{a=a(s)}
  \norm{p-q-a(s)}_{L^{3/2}(Q_\vartheta)}.
\]
\end{definition}

\begin{assumption}[Field control by the audit baseline norm]
\label{ass:field-control-audit}
Each audit representative \(\zeta\in\Gamma_{\Lambda,k}^{\intg}\) used in this
comparison has associated fields \((u_\zeta,p_\zeta)\) on \(Q_1\), modulo
time-dependent pressure constants, and there is a finite-window constant
\(C_{\mathrm{fld}}\) such that
\[
  d_\vartheta^{\mathrm{fld}}\bigl((u_k,p_k),(u_\zeta,p_\zeta)\bigr)
  \le
  C_{\mathrm{fld}}\lvert D_k-\zeta\rvert_{0,k}
\]
for every such representative.
\end{assumption}

\begin{assumption}[CKN cleaning of near-minimizing audit representatives]
\label{ass:ckn-cleaning}
For every \(\eta>0\) there exists
\(\zeta_\eta\in\Gamma_{\Lambda,k}^{\intg}\) with
\[
  \lvert D_k-\zeta_\eta\rvert_{0,k}\le\delta_k+\eta
\]
and a CKN-clean representative
\(\zeta_\eta^{\mathrm{CKN}}\in\Gamma_{\Lambda,k}^{\mathrm{CKN}}\) such that
\[
  d_\vartheta^{\mathrm{fld}}
  \bigl((u_{\zeta_\eta},p_{\zeta_\eta}),
  (u_{\zeta_\eta^{\mathrm{CKN}}},
  p_{\zeta_\eta^{\mathrm{CKN}}})\bigr)
  \le
  \Delta_k^{\mathrm{gauge}}+\eta.
\]
\end{assumption}

\begin{proposition}[Metric realization of the CKN-compatible domination]
\label{prop:metric-realizes-ckn-domination}
Assume \cref{ass:field-control-audit,ass:ckn-cleaning}.  Then
\cref{ass:ckn-audit-domination} holds with
\[
  C_{\mathrm{CKN/aud}}=C_{\mathrm{fld}}.
\]
\end{proposition}

\begin{proof}
Fix \(\eta>0\), and choose \(\zeta_\eta\) and
\(\zeta_\eta^{\mathrm{CKN}}\) as in \cref{ass:ckn-cleaning}.  By the
definition of \(\delta_k^{\mathrm{CKN}}\),
\[
  \delta_k^{\mathrm{CKN}}
  \le
  d_\vartheta^{\mathrm{fld}}
  \bigl((u_k,p_k),
  (u_{\zeta_\eta^{\mathrm{CKN}}},
  p_{\zeta_\eta^{\mathrm{CKN}}})\bigr).
\]
The triangle inequality for \(d_\vartheta^{\mathrm{fld}}\) gives
\[
  \delta_k^{\mathrm{CKN}}
  \le
  d_\vartheta^{\mathrm{fld}}
  \bigl((u_k,p_k),(u_{\zeta_\eta},p_{\zeta_\eta})\bigr)
  +
  d_\vartheta^{\mathrm{fld}}
  \bigl((u_{\zeta_\eta},p_{\zeta_\eta}),
  (u_{\zeta_\eta^{\mathrm{CKN}}},
  p_{\zeta_\eta^{\mathrm{CKN}}})\bigr).
\]
For the pressure component, this triangle inequality is taken in the quotient
by time-dependent constants: if two differences are represented using gauges
\(a_1(s)\) and \(a_2(s)\), then their sum is represented using the admissible
gauge \(a_1(s)+a_2(s)\), and taking infima gives the displayed estimate.
Using \cref{ass:field-control-audit,ass:ckn-cleaning}, this is bounded by
\[
  C_{\mathrm{fld}}(\delta_k+\eta)
  +
  \Delta_k^{\mathrm{gauge}}+\eta.
\]
Letting \(\eta\downarrow0\) yields
\[
  \delta_k^{\mathrm{CKN}}
  \le
  C_{\mathrm{fld}}\delta_k+\Delta_k^{\mathrm{gauge}},
\]
which is \cref{ass:ckn-audit-domination} with
\(C_{\mathrm{CKN/aud}}=C_{\mathrm{fld}}\).
\end{proof}

\begin{corollary}[Built-in CKN-compatible audit metric]
\label{cor:built-in-ckn-metric}
Suppose the audit clean target for this branch is already the
CKN-compatible class,
\[
  \Gamma_{\Lambda,k}^{\intg}=\Gamma_{\Lambda,k}^{\mathrm{CKN}},
\]
and the audit baseline norm controls the CKN field distance as in
\cref{ass:field-control-audit}.  Then
\[
  \delta_k^{\mathrm{CKN}}\le C_{\mathrm{fld}}\delta_k.
\]
Thus \cref{ass:ckn-audit-domination} holds with
\(\Delta_k^{\mathrm{gauge}}=0\) and
\(C_{\mathrm{CKN/aud}}=C_{\mathrm{fld}}\).
\end{corollary}

\begin{proof}
In this case take
\(\zeta_\eta^{\mathrm{CKN}}=\zeta_\eta\) in
\cref{ass:ckn-cleaning}.  The cleaning error is zero, and
\cref{prop:metric-realizes-ckn-domination} gives the result.
\end{proof}

\begin{remark}[Status of the metric realization]\label{rem:metric-realization-status}
\Cref{prop:metric-realizes-ckn-domination,cor:built-in-ckn-metric} are
finite-window norm-comparison statements.  They do not prove that an arbitrary
audit metric controls the CKN field distance, and they do not prove that an
arbitrary audit representative can be cleaned into
\(\Gamma_{\Lambda,k}^{\mathrm{CKN}}\).  They identify the exact local metric
and representative-cleaning inputs needed for the CKN defect-extraction branch.
\end{remark}

\subsection{Quotient closeness implies CKN smallness}

\begin{lemma}[CKN smallness from CKN-compatible quotient closeness]
\label{lem:ckn-closeness-smallness}
Assume there exists \(\zeta\in\Gamma_{\Lambda,k}^{\mathrm{CKN}}\) such that
\[
  \norm{u_k-u_\zeta}_{L^3(Q_\vartheta)}
  +
  \inf_{a=a(s)}
  \norm{p_k-p_\zeta-a(s)}_{L^{3/2}(Q_\vartheta)}
  \le
  \tau.
\]
Then there is a constant \(C_\vartheta\) such that
\[
  \Phi_k(\vartheta)
  \le
  2\Phi_\zeta(\vartheta)
  +
  C_\vartheta(\tau^3+\tau^{3/2}).
\]
In particular, since
\(\Phi_\zeta(\vartheta)\le\varepsilon_{\mathrm{CKN}}/4\), one also has the
coarser bound
\[
  \Phi_k(\vartheta)
  \le
  C_\vartheta
  \left[
  \frac{\varepsilon_{\mathrm{CKN}}}{4}
  +
  \tau^3+\tau^{3/2}
  \right].
\]
Consequently there exists \(\tau_{\mathrm{CKN}}>0\) such that
\[
  \tau\le\tau_{\mathrm{CKN}}
  \quad\Longrightarrow\quad
  \Phi_k(\vartheta)<\varepsilon_{\mathrm{CKN}}.
\]
\end{lemma}

\begin{proof}
Fix \(\eta_0>0\).  Since the pressure distance contains an infimum over
time-dependent constants, choose \(a(s)\) such that
\[
  \norm{u_k-u_\zeta}_{L^3(Q_\vartheta)}
  +
  \norm{p_k-p_\zeta-a(s)}_{L^{3/2}(Q_\vartheta)}
  \le
  \tau+\eta_0.
\]
For every \(r>1\) and every \(\alpha>0\),
\[
  |A+B|^r\le (1+\alpha)|A|^r+C_{r,\alpha}|B|^r .
\]
Using this with \(r=3\), \(A=u_\zeta\), and
\(B=u_k-u_\zeta\), and choosing \(\alpha\le1\), gives
\[
  \int_{Q_\vartheta}|u_k|^3
  \le
  2\int_{Q_\vartheta}|u_\zeta|^3
  +
  C\norm{u_k-u_\zeta}_{L^3(Q_\vartheta)}^3.
\]

For pressure, write
\[
  q:=p_k-p_\zeta-a(s).
\]
Since spatial mean subtraction removes the time-dependent pressure constant,
\[
  p_k-(p_k)_{B_\vartheta}(s)
  =
  p_\zeta-(p_\zeta)_{B_\vartheta}(s)
  +
  q-(q)_{B_\vartheta}(s).
\]
Applying the same perturbative inequality with \(r=3/2\) yields
\[
  \int_{Q_\vartheta}|p_k-(p_k)_{B_\vartheta}(s)|^{3/2}
  \le
  2\int_{Q_\vartheta}
  |p_\zeta-(p_\zeta)_{B_\vartheta}(s)|^{3/2}
  +
  C\int_{Q_\vartheta}|q-(q)_{B_\vartheta}(s)|^{3/2}.
\]
The mean-subtraction map is bounded on \(L^{3/2}(B_\vartheta)\), uniformly for
the fixed ball \(B_\vartheta\), so
\[
  \int_{Q_\vartheta}|q-(q)_{B_\vartheta}(s)|^{3/2}
  \le
  C_\vartheta\norm{q}_{L^{3/2}(Q_\vartheta)}^{3/2}.
\]
Combining the velocity and pressure estimates and using the choice of
\(a(s)\) gives
\[
  \Phi_k(\vartheta)
  \le
  2\Phi_\zeta(\vartheta)
  +
  C_\vartheta\bigl((\tau+\eta_0)^3+(\tau+\eta_0)^{3/2}\bigr).
\]
Letting \(\eta_0\downarrow0\) proves the first estimate.  The second estimate
follows from \(\Phi_\zeta(\vartheta)\le\varepsilon_{\mathrm{CKN}}/4\).

Choose \(\tau_{\mathrm{CKN}}>0\) so small that
\[
  C_\vartheta
  \left(
  \tau_{\mathrm{CKN}}^3+\tau_{\mathrm{CKN}}^{3/2}
  \right)
  <
  \frac{\varepsilon_{\mathrm{CKN}}}{2}.
\]
Then the first estimate gives
\[
  \Phi_k(\vartheta)
  <
  \frac{\varepsilon_{\mathrm{CKN}}}{2}
  +
  \frac{\varepsilon_{\mathrm{CKN}}}{2}
  =
  \varepsilon_{\mathrm{CKN}},
  \qquad
  \tau\le\tau_{\mathrm{CKN}}.
\]
\end{proof}

\begin{theorem}[Audit closeness implies CKN smallness]
\label{thm:audit-closeness-ckn}
Assume the CKN-compatible audit domination in
\cref{ass:ckn-audit-domination}.  If
\[
  C_{\mathrm{CKN/aud}}\,\delta_k+\Delta_k^{\mathrm{gauge}}
  \le
  \tau_{\mathrm{CKN}},
\]
then
\[
  \Phi_k(\vartheta)<\varepsilon_{\mathrm{CKN}}.
\]
\end{theorem}

\begin{proof}
By \cref{ass:ckn-audit-domination},
\[
  \delta_k^{\mathrm{CKN}}\le\tau_{\mathrm{CKN}}.
\]
For every \(\eta_0>0\), the definition of
\(\delta_k^{\mathrm{CKN}}\) gives a representative
\(\zeta_{\eta_0}\in\Gamma_{\Lambda,k}^{\mathrm{CKN}}\) whose
CKN-compatible distance to \(D_k\) is at most
\(\delta_k^{\mathrm{CKN}}+\eta_0\).  Applying
\cref{lem:ckn-closeness-smallness} in its quantitative form gives
\[
  \Phi_k(\vartheta)
  \le
  \frac{\varepsilon_{\mathrm{CKN}}}{2}
  +
  C_\vartheta
  \left[
  (\delta_k^{\mathrm{CKN}}+\eta_0)^3
  +
  (\delta_k^{\mathrm{CKN}}+\eta_0)^{3/2}
  \right].
\]
Letting \(\eta_0\downarrow0\) and using
\(\delta_k^{\mathrm{CKN}}\le\tau_{\mathrm{CKN}}\), together with the defining
choice of \(\tau_{\mathrm{CKN}}\), gives
\(\Phi_k(\vartheta)<\varepsilon_{\mathrm{CKN}}\).
\end{proof}

\begin{corollary}[CKN-bad scale forces baseline audit defect]
\label{cor:ckn-bad-defect}
Assume \cref{ass:ckn-audit-domination}.  If
\[
  \Phi_k(\vartheta)\ge\varepsilon_{\mathrm{CKN}},
\]
then
\[
  C_{\mathrm{CKN/aud}}\,\delta_k+\Delta_k^{\mathrm{gauge}}
  >
  \tau_{\mathrm{CKN}}.
\]
In particular, if
\[
  \Delta_k^{\mathrm{gauge}}\le \frac12\tau_{\mathrm{CKN}},
\]
then
\[
  \delta_k\ge\delta_*,
  \qquad
  \delta_*:=
  \frac{\tau_{\mathrm{CKN}}}{2C_{\mathrm{CKN/aud}}}.
\]
\end{corollary}

\begin{proof}
The first conclusion is the contrapositive of
\cref{thm:audit-closeness-ckn}.  If additionally
\(\Delta_k^{\mathrm{gauge}}\le\tau_{\mathrm{CKN}}/2\), then
\[
  C_{\mathrm{CKN/aud}}\delta_k
  >
  \tau_{\mathrm{CKN}}-\Delta_k^{\mathrm{gauge}}
  \ge
  \frac12\tau_{\mathrm{CKN}},
\]
which gives the stated lower bound for \(\delta_k\).
\end{proof}

\subsection{Conditional finite audit-chain extraction}

\begin{theorem}[Conditional finite audit-chain extraction from CKN-bad scales]
\label{thm:conditional-ckn-chain-extraction}
Let \((u,p)\) be a suitable weak solution and fix a finite geometric scale
chain
\[
  r_k=\lambda^k r_0,\qquad k=0,\ldots,K.
\]
Let \(D_k(u,p)\) be the NS-generated package at scale \(r_k\).  Assume:
\begin{enumerate}[label=(\alph*)]
  \item the scale \(r_k\) is CKN-bad for each \(k=0,\ldots,K\):
  \[
    \Phi_k(\vartheta)\ge\varepsilon_{\mathrm{CKN}};
  \]
  \item the broad one-step admissibility theorem applies to each transition
  \(D_k\mapsto D_{k+1}\);
  \item the CKN-compatible audit domination
  \cref{ass:ckn-audit-domination} holds at each considered scale;
  \item the gauge comparison errors satisfy
  \[
    \Delta_k^{\mathrm{gauge}}\le\frac12\tau_{\mathrm{CKN}},
    \qquad k=0,\ldots,K.
  \]
\end{enumerate}
Then
\[
  \delta_k\ge\delta_*>0,
  \qquad k=0,\ldots,K.
\]
Consequently
\[
  D_0\to D_1\to\cdots\to D_K
\]
is an admissible finite audit chain with nondegenerate baseline defect, and
the finite recursive anti-phantom theorem may be applied to this finite chain
whenever its other hypotheses are satisfied.
\end{theorem}

\begin{proof}
For each \(k\), CKN-badness and the gauge-error bound allow
\cref{cor:ckn-bad-defect} to be applied at that scale.  This gives
\(\delta_k\ge\delta_*\) for all \(k=0,\ldots,K\).  The broad one-step
admissibility theorem supplies admissibility of the transitions in the finite
chain.  The last sentence is therefore an insertion into the already
established finite recursive audit theorem, not a new regularity conclusion.
\end{proof}

\begin{remark}[CKN singular candidates]\label{rem:ckn-singular-candidates}
If \(z_0\) is a CKN singular candidate in the standard epsilon-regularity
sense \cite{CKN1982,Lin1998}, then CKN epsilon regularity implies that sufficiently small normalized
scales centered at \(z_0\) are CKN-bad: otherwise a scale with
\(\Phi_k(\vartheta)<\varepsilon_{\mathrm{CKN}}\) would give regularity on a
smaller cylinder.  \Cref{thm:conditional-ckn-chain-extraction} uses only this
finite-chain CKN-badness input and does not produce a contradiction or a
regularity conclusion.
\end{remark}

\begin{remark}[What remains open]\label{rem:ckn-open}
The genuinely hard PDE input in this branch is the domination
\[
  \delta_k^{\mathrm{CKN}}
  \le
  C_{\mathrm{CKN/aud}}\,\delta_k+\Delta_k^{\mathrm{gauge}}.
\]
If the audit baseline distance is defined to include the CKN-compatible
\(L^3\) velocity and mean-free \(L^{3/2}\) pressure components, this
comparison is built into the metric.  If the existing audit distance is
weaker, the comparison must be proved separately.  The branch proves only the
finite-window implication
\[
  \text{CKN-bad scale}
  \Longrightarrow
  \text{positive audit defect}
\]
under a CKN-compatible quotient convention.  It does not prove detector/ledger
cost unsustainability, a CKN-small scale conclusion, Navier--Stokes
regularity, singularity exclusion, or any Clay-level theorem.
\end{remark}

\section{Final Finite-Window Theorem}
\label{sec:final-finite-window-theorem}

At each scale, the finite-window coefficient has the schematic form
\[
  c_k=\mu_{\Lambda,k}^{\comp}\lambda_{G,k}-L_k^{\mathrm{res}},
\]
where \(\mu_{\Lambda,k}^{\comp}\) is a clean compact quotient gap, \(\lambda_{G,k}\) is chart visibility, and \(L_k^{\mathrm{res}}\) is the residual-loss coefficient.  In the pressure--flux--energy realization one may take \(\mu_{\Lambda,k}^{\comp}=\mu_{\Lambda,k}^{\mathrm{PFE}}\), provided the reduced matrix condition
\[
  \ker T_{\Lambda,k}^{\mathrm{PFE}}=G_{\Lambda,k}^{\cl}
\]
holds on the selected clean coordinates.

\begin{theorem}[NS-generated recursive finite-window audit chain]\label{thm:final-fw-chain}
Fix a finite geometric scale chain \(D_0\to D_1\to\cdots\to D_K\) generated by normalized local Navier--Stokes data.  Assume the following.
\begin{enumerate}[label=(\alph*)]
  \item The broad one-step admissibility theorem, \cref{thm:main}, applies to every transition \(D_k\to D_{k+1}\), with ledger
  \[
    \Delta_k=\Delta_k^{\sync}+\Delta_k^{\loc,\EF}+\Delta_k^{\proj}+\Delta_k^{\harm}+\Delta_k^{\chart}+\Delta_k^{\gap}+\Delta_k^{\gate}+\Delta_k^{\detc}.
  \]
  \item The static finite-window audit certificate \eqref{eq:scale-audit-cert} holds at every scale.
  \item The recursive error and coefficient update rules in \cref{ass:recursive-update-rules} hold.
  \item The coefficient drift is controlled by \cref{thm:coeff-update}, or equivalently the accumulated coefficient loss keeps \(c_K^{\min}>0\) on the considered finite chain.
  \item Projection-tail control is supplied either by compactness, by effective projection, or by the drift estimate in \cref{thm:compact-image-proj,prop:effective-proj,thm:proj-tail-drift}.
  \item Chart visibility and clean detector gap are supplied by the compact quotient criteria in \cref{thm:chart-vis,thm:clean-gap}; in the PFE realization the clean gap is supplied by \cref{thm:pfe-clean-gap,thm:pfe-smred-gap}.
\end{enumerate}
Then, for every choice of weights \(w_k\ge0\),
\[
  \sum_{k=0}^K w_kM_k
  \ge
  c_K^{\min}\sum_{k=0}^K w_k\delta_k
  -E_K^{\mathrm{rec}},
  \qquad
  E_K^{\mathrm{rec}}:=\sum_{k=0}^K w_kE_k.
\]
The individual \(E_k\) are bounded by the variable-coefficient recursion in \cref{lem:var-recursion}.  In particular, either
\[
  \sum_{k=0}^K w_kM_k
  \ge
  \frac{c_K^{\min}}2\sum_{k=0}^K w_k\delta_k,
\]
or
\[
  E_K^{\mathrm{rec}}
  \ge
  \frac{c_K^{\min}}2\sum_{k=0}^K w_k\delta_k.
\]
\end{theorem}

\begin{proof}
By \cref{thm:main}, each NS-generated transition is admissible after every named mismatch is charged to the ledger.  By \cref{thm:one-step-audit-propagation}, the static certificate is therefore available at each scale with recursively controlled coefficient and residual budget.  The weighted lower bound is exactly \cref{thm:finite-chain}; the alternative is \cref{cor:recursive-alt}.  The listed structural assumptions identify where projection, chart, clean-gap, PFE, and coefficient-loss inputs enter the ledger.
\end{proof}

\begin{corollary}[Smooth reduced selected-class chain]\label{cor:smred-final-chain}
Let
\[
  D_0\to D_1\to\cdots\to D_K\in\A_{K,N,M}^{\mathrm{sm,red}}.
\]
Assume the selected structural inputs in \cref{thm:smred-selected-inputs} hold and that the remaining synchronization, harmonic truncation, gate/slack, detector, residual, coefficient-drift, projection-transport, and chart-transport terms are charged to their ledger components.  Then the finite-chain lower bound and alternative in \cref{thm:final-fw-chain} hold for this smooth reduced NS-generated chain.
\end{corollary}

\begin{proof}
The selected pressure/source, localization, projection, chart, and PFE clean-gap inputs are supplied by \cref{thm:smred-selected-inputs}.  The remaining assumptions are exactly the ledger assignments stated in the corollary.  Thus \cref{thm:final-fw-chain} applies.
\end{proof}

\section{Scope, Open Obligations, and Limitations}
\label{sec:limitations}

The final theorem deliberately stops at finite-window and finite-chain statements.  The following items remain outside the scope of the results.
\begin{enumerate}[label=(\roman*)]
  \item proving one-step admissibility for all suitable weak solutions without structural ledger hypotheses;
  \item proving uniform positive lower bounds for \(c_k\) as \(k\to\infty\);
  \item proving summability of recursive increments in arbitrary cascades;
  \item deriving reduced chart kernel-freeness and clean detector kernel-freeness for broad infinite-dimensional NS-generated package classes;
  \item proving the audit-to-CKN domination hypothesis when it is not built into the baseline metric;
  \item proving detector/ledger cost unsustainability or a CKN-small scale conclusion;
  \item proving Navier--Stokes regularity, singularity exclusion, or any Clay-level theorem.
\end{enumerate}

\begin{remark}[Final positioning]
The one-step admissibility part supplies the left edge of the recursive audit framework: it explains how a single NS-generated package can remain admissible after restriction, rescaling, synchronization, and recomputation when every mismatch is charged.  The recursive part supplies the finite-chain middle mechanism: once admissible packages and static certificates are available at each scale, the detector lower bound propagates through a weighted finite chain.  Any infinite-chain or regularity-facing conclusion requires additional uniform hypotheses not proved here.
\end{remark}

\appendix

\section{Renormalized Navier--Stokes Scaling}

The scaling
\[
  u_k(y,s)=r_k u(x_0+r_ky,t_0+r_k^2s),\qquad
  p_k(y,s)=r_k^2 p(x_0+r_ky,t_0+r_k^2s)
\]
preserves the form of the incompressible Navier--Stokes equations.  The scale-critical quantities used in the audit packages are always computed in the normalized variables.  This convention prevents hidden powers of \(r_k\) from entering the recursive ledger.

\section{Relation to the Static Finite-Window Theorem}

The static finite-window theorem \eqref{eq:static-audit-fw} is used as the audit input at each scale from the finite-window audit and computational anti-phantom framework \cite{YuSingularityAuditTransfer2026,YuComputationalAntiPhantom2026}.  The results above do not reprove that static theorem.  Instead, they combine broad NS-generated one-step admissibility with finite-chain recursive summation, while isolating all structural inputs needed for projection, chart visibility, clean detector gaps, PFE kernel-freeness, and CKN-compatible defect extraction.

\end{document}